\documentclass{article}

\usepackage{amsmath,amssymb,amsthm,latexsym,enumerate}
\usepackage[utf8]{inputenc}
\usepackage{a4wide}

\newtheorem{theorem}{Theorem}[section]

\newtheorem{claim}[theorem]{Claim}
\newtheorem{proposition}[theorem]{Proposition}

\newtheorem{corollary}[theorem]{Corollary}

\renewenvironment{proof}{\noindent{\bfseries Proof\,}}{\hfill$\Box$}

\makeatletter
\newenvironment{proofof}[1]{\par
  \pushQED{\qed}%
  \normalfont \topsep6\p@\@plus6\p@\relax
  \trivlist
  \item[\hskip\labelsep
        \bfseries
    Proof of #1\@addpunct{.}]\ignorespaces
}{%
  $\quad$\hfill$\Box$\endtrivlist\@endpefalse
}
\makeatother

\def \C   	{{\cal C}}
\def \End   	{\text{End}}
\def \F   	{{\cal F}}
\def \G   	{{\cal G}}
\def \ham	{{\mathcal H}_n} 
\def \M   	{{\cal M}}
\def \N		{\mathbb{N}}
\def \P		{{\cal P}}
\def \S		{{\cal S}}
\def \T   	{{\cal T}}
\def \pkf	{{\P_{k,n}}} % P_k-factor
\def \pm	{{\M}_n} % perfect matching

\def \skf	{{\S_{k,n}}} % S_k-factor

\title{How fast can Maker win in fair biased games?
\thanks{The research was initiated at FU Berlin, during the research visit of Mirjana Mikala\v{c}ki supported by DAAD.}}% We gratefully acknowledge the support.}}
\author{Dennis Clemens  \thanks {Email: dennis.clemens@tuhh.de}\\
 \small{ \textit{Technische Universit\"at Hamburg-Harburg, Institut f\"ur Mathematik,}}\\[-0.8ex]
 \small{\textit{Am Schwarzenberg-Campus 3, 21073 Hamburg, Germany.}}\\
  \and Mirjana Mikala\v{c}ki \thanks{Research partly supported by Ministry of Science and Technological Development, Republic of Serbia, and Provincial Secretariat for
Science, Province of Vojvodina. Email: mirjana.mikalacki@dmi.uns.ac.rs}\\
\small{ \textit{Department of Mathematics and Informatics, Faculty of Sciences,}}\\[-0.8ex]
 \small{\textit{ University of
Novi Sad, Serbia.}}} 
\date{}

\begin{document}
\maketitle

\begin{abstract}
We study $(a:a)$ Maker-Breaker games played on the edge set of the complete graph on $n$ vertices. In the following four games -- perfect matching game, Hamilton cycle game, star factor game and path factor game, our goal is to determine the least number of moves which Maker needs in order to win these games. Moreover, for all games except for the star factor game, we show how Red can win in the strong version of these games.

\end{abstract}

\section{Introduction}
Let $a$ and $b$ be two positive integers, let $X$ be a finite set and $\F \subseteq 2^X$ be a family of the subsets of $X$. In the $(a:b)$ \textit{positional game} $(X,\F)$, two players take turns in claiming $a$, respectively $b$, previously unclaimed elements of $X$, with one of them going first. The set $X$ is referred to as the \textit{board} of the game, while the elements of $\F$ are referred to as the \textit{winning sets}. When there is no risk of confusion on which board the game is played, we just use $\F$ to denote the game. The integers $a$ and $b$ are referred to as \textit{biases} of the players. When $a=b$, the game is said to be \textit{fair}. If $a=b=1$, the game is called \textit{unbiased}. Otherwise, the game is called \textit{biased}. If a player has a strategy to win against any strategy of the other player, this strategy is called \textit{a winning strategy}.%, and we say that this player \textit{wins} the game.

In the $(a:b)$ \textit{Maker-Breaker} positional game $(X,\F)$, the two players are called \textit{Maker} and \textit{Breaker}. Maker wins the game $\F$ at the moment she claims all the elements of some $F\subseteq \F$. If Maker did not win by the time all the elements of $X$ are claimed by some player, then Breaker wins the game $\F$. In order to show that Maker wins the game as both first and second player, we will assume in this paper that Breaker starts the game (as being the first player can only be an advantage in Maker-Breaker games).
%In order to show that Maker has a winning strategy as both first and second player, in this paper  we will assume that Breaker starts the game (as being the first player cannot be a disadvantage in Maker-Breaker games).
%In this paper we will assume that Breaker starts the game. %In each turn, Breaker claims $b$ unclaimed elements and then Maker claims $a$ unclaimed elements of $X$. 

It is very natural to play Maker-Breaker games on the edge set of a given graph $G$ (see e.g.~\cite{Beckbook,HKSSBook}). %Such games have been studied quite a lot in the recent years.
Here, we focus on the $(a:b)$ games played on the edge set of the complete graph on $n$ vertices, $K_n$, where $n$ is a sufficiently large integer. That is, in this case the board is $X=E(K_n)$.% and $\F$ consists of predefined graph theoretic structures.

For example: in the \textit{connectivity game}, $\T_n$, the winning sets are all spanning trees of $K_n$; in the \textit{perfect matching game}, $\pm$, the winning sets are all independent edge sets of size $\lfloor n/2 \rfloor$ (note that in case $n$ is odd, this matching covers all but one vertex in $K_n$); in the \textit{Hamilton cycle game}, $\ham$, the winning sets are all Hamilton cycles of $K_n$; in the $k$-vertex-connectivity game, $\C^k_n$, for $k\in \N$, the winning sets are all $k$-vertex-connected graphs on $n$ vertices. 

It is not very difficult to see that Maker wins all aforementioned unbiased games. Therefore, we can ask the following question: \textit{How quickly can Maker win the game?} With parameter $\tau_{\F}(a:b)$ we denote the shortest \textit{duration} of the $(a:b)$ Maker-Breaker game $\F$, i.e.\ the least number of moves $t$ such that Maker has a strategy to win the $(a:b)$ game $\F$ within $t$ moves. For completeness, we say that $\tau_{\F}(a:b)=\infty$ if Breaker has a winning strategy.

It was shown in~\cite{Lehman} that, for $n\geq 4$, $\tau_{\T_n}(1:1)=n-1$, which is optimal. In~\cite{HKSS09} it was proved that $\tau_{\pm}(1:1)=n/2+1$, when $n$ is even, and $\tau_{\pm}(1:1)=\lceil n/2\rceil$, when $n$ is odd and also that $\tau_{\ham}(1:1)\leq n+2$ and $\tau_{\C^k_n}(1:1)=kn/2+o(n)$. 
Hefetz and Stich in~\cite{HS09} showed that $\tau_{\ham}(1:1)=n+1$, and Ferber and Hefetz~\cite{FH} recently showed that $\tau_{\C^k_n}(1:1)=\lfloor kn/2\rfloor+1$. Moreover, there are corresponding results when the mentioned games are played on graphs that are not complete, see e.g.~\cite{CFKL}.

In this paper, we are particularly interested in $(a:a)$ Maker-Breaker games on $E(K_n)$, for constant $a\geq 1$. Although these games are studied less than unbiased and $(1:b)$ games, they are also significant. Just a slight change in the bias from $a=1$ to $a=2$ can completely change the outcome (and thus the course of the play) of some games (see~\cite{Beckbook}). One example is the \textit{diameter-2} game (where the winning sets are all graphs with diameter at most 2). It was proved in~\cite{BMP} that Breaker wins the $(1:1)$ \textit{diameter-2} game, but Maker wins the $(2:2)$ \textit{diameter-2} game.

Not so much is known about fast winning strategies in fair $(a:a)$ Maker-Breaker games, where $a$ can be greater than $1$. From the results in~\cite{GeSa,HMS12}, we obtain that in the connectivity game $\tau_{\T_n}(a:a)=\lceil(n-1)/a\rceil$.

Our research is concentrated on fast winning strategies in four $(a:a)$ Maker-Breaker games, for $a\in \N$.
Firstly, we take a look at the $(a:a)$ perfect matching game, $\pm$. The case $a=1$ is already proved in~\cite{HKSS09}, and we show the following theorem for all $a\geq 2$.
\begin{theorem}\label{WeakPM}
Let $a\in\N.$ Then for every large enough $n$
the following is true for the $(a:a)$ Maker-Breaker perfect matching game:
\begin{align*}
\tau_{\pm}(a:a)=
\begin{cases}
 \frac{n}{2a}  + 1, 
		& \text{ if  $a= 1$ and $n$ is even,}\\[0.2cm]
\left\lceil \frac{n}{2a} \right\rceil - 1, 
		& \text{ if  $2a \mid n-1$}\\[0.2cm]
\left\lceil \frac{n}{2a} \right\rceil, 
		& \text{ otherwise.}
\end{cases}
\end{align*}
\end{theorem}

\medskip

Secondly, we analyse the $(a:a)$ Maker-Breaker Hamilton cycle game, $\ham$, and prove the following result for $a\geq 2$. The case $a=1$ is proved in~\cite{HS09}. 

\begin{theorem}\label{WeakHam}
Let $a\in\N.$ Then for every large enough $n$
the following is true for the $(a:a)$ Maker-Breaker Hamilton cycle game:
\begin{align*}
\tau_{\ham}(a:a)=
\begin{cases}
 \frac{n}{a}  + 1, 	
		& \text{ if } a=1 \text { or } (a=2 \text { and } n \text{ is even}),\\[0.2cm]
\left\lceil \frac{n}{a} \right\rceil, 		
		& \text{ otherwise.}
\end{cases}
\end{align*}
\end{theorem}

\medskip

We study two more $(a:a)$ Maker-Breaker games whose winning sets are spanning graphs. More precisely, we are interested in factoring the graph $K_n$ with stars and paths. For fixed $k\geq 2$, let $P_k$ denote a path with $k$ vertices, and let $S_k$ denote a star with $k-1$ leaves. Now, for all large enough $n$, such that $k \mid n$, we are interested in finding winning strategies in the $(a:a)$ $P_k$-factor game, denoted by $\pkf$, and in the $(a:a)$ $S_k$-factor game, denoted by $\skf$, where the winning sets are all path factors and star factors of $K_n$, respectively, on $k$ vertices. We show the following.
\begin{theorem}
\label{weakPk}
Let $a\in\N$ and $k\in\N.$ Then for every large enough $n$, such that $k \mid n$, the following is true for the $(a:a)$ Maker-Breaker $P_k$-factor game:
\begin{align*}
\tau_{\pkf}(a:a)= \left\lceil \frac{(k-1)n}{ka} \right\rceil . 	
\end{align*}
\end{theorem}

\medskip

%The following theorem shows how quickly Maker can win in the $(a:a)$ $S_k$ factor game.
\medskip

\begin{theorem}
\label{weakSk}
Let $a\geq 1$ and $k\geq 3$ be integers. Then for every large enough $n$, such that $k \mid n$, the following is true for the
$(a:a)$ Maker-Breaker $S_k$-factor game:
%\begin{align*}
%\left\lceil \frac{kn}{(k+1)a} \right\rceil \leq \tau_{\skf}(a:a)\leq \left\lceil \frac{kn}{(k+1)a} \right\rceil +1.
%\end{align*}
\begin{align*}
\tau_{\skf}(a:a)\leq
\begin{cases}
\left\lceil \frac{(k-1)n}{ka} \right\rceil,		
		& \text{ if } a\nmid \frac{(k-1)n}{k} ,\\[0.2cm]
 \frac{(k-1)n}{ka}  +1, 		
		& \text{ otherwise.}
\end{cases}
\end{align*}

\end{theorem}

\textbf{Strong games.} We also look at another type of positional games. In the \textit{strong} positional game $(X,\F)$, the two players are called \textit{Red} and \textit{Blue}, and Red starts the game. The winner of the game is the \textit{first} player who claims all the elements of one $F \in \F$. If none of the players manage to do that before all the elements of $X$ are claimed, the game ends in a \textit{draw}. 

%It is well-known from the classic Game theory that Blue cannot have the winning strategy in the strong game. 
By the \textit{strategy stealing argument} (see~\cite{Beckbook}), Blue cannot have a winning strategy in the strong game. So, in every strong game, either Red wins, or Blue has a drawing strategy. For the games where the draw is impossible, we know that Red wins. Unfortunately, the existence of Red's strategy tells us nothing about how Red should play in order to win. Finding explicit winning strategies for Red can be very difficult. The results in~\cite{FH,FHcon} show that fast winning strategies for Maker in certain games can be used in order to describe the winning strategies for Red in the strong version of these games. 

If Maker can win \textit{perfectly fast} in the $(a:a)$ game $\F$, i.e.\ if the number of moves, $t$, she needs to win is equal to $\lceil \min(|F|:F\in\F)/a \rceil $, that immediately implies Red's win in the strong game. Indeed, as Red starts the game, Blue has no chance to fully claim any winning set in less than $t$ moves. Thus, Red can play according to the strategy of Maker, without worrying about Blue's moves, by which Red will claim a winning set in $t$ moves, thus winning the game. 

From Theorem~\ref{WeakPM}, we see that Maker can win perfectly fast in the $(a:a)$ perfect matching game %(as all winning sets are of size $n/2$) 
in all cases, but in case $a=1$. Therefore, we immediately see that for $a\neq 1$, Red has a winning strategy for the corresponding strong game. For $a=1$ the proof that Red wins the strong game appears in \cite{FH}. 
Similarly to the perfect matching game, from Theorem~\ref{WeakHam} we can immediately see that Red has a winning strategy for the strong $(a:a)$ Hamilton cycle game in all but two cases -- the case $a=1$ and the case $a=2$ and $n$ is even. The case $a=1$ appears in \cite{FH}, and for the remaining case, we prove the following theorem.

\begin{theorem}\label{StrongHam}
For every large enough even $n$ the following is true:
Red has a strategy for the $(2:2)$ Hamilton cycle game to win within 
$n/2+1$ rounds. 
\end{theorem}

Every $P_k$-factor of $K_n$, for given $k\in \N$ such that $k \mid n$, has to have $n(k-1)/k$ edges. Therefore, from Theorem~\ref{weakPk}, we obtain that Maker can win perfectly fast in $(a:a)$ game $\pkf$ and Red can use the winning strategy of Maker in this game to win in the corresponding strong game.\\

{\bf Notation and terminology.}
Our graph-theoretic notation is standard and mostly follows that of~\cite{West}. 
In par\-ti\-cu\-lar, we use the following. We write $[n]:=\{1,2,\ldots,n\}$. 
For a graph $G$, $V(G)$ and $E(G)$ denote its sets of vertices and edges respectively, $v(G) = |V(G)|$ and $e(G) = |E(G)|$. For disjoint sets $A,B \subseteq V(G)$, let $E_G(A,B)$ denote the set of edges of $G$ with one endpoint in $A$ and one endpoint in $B$, and $e_G(A,B):=|E_G(A,B)|$. Moreover, $E_G(A):=E_G(A,A)$ and $e_G(A):=|E_G(A)|$. For a vertex $x \in V(G)$ and a set $S\subseteq V(G)$, $N_G(x,S) = \{u \in
S : ux \in E(G)\}$ denotes the set of {\em neighbours} of the vertex $x$ in the set $S$
with respect to (w.r.t.) $G$. We set $N_G(x):=N_G(x,V(G))$. Moreover, $d_G(x,S):=|N_G(x,S)|$
denotes the {\em degree} of $x$ into $S$, while $d_G(x) = |N_G(x)|$ denotes 
the {\em degree} of $x$ in the graph $G$.
Whenever there is no risk of confusion, we omit the subscript $G$ in the notation above.
Given a graph $G=(V,E)$, we let $\overline{G}=(V,\overline{E})$ denote its {\em complement},
where $\overline{E}:=\{xy\notin E:\ x,y\in V\}$. For $e\in E(G)$, we set $G-e:=(V,E\setminus \{e\})$.
For $S\subseteq V$, $G[S]$ denotes the subgraph of $G$ {\em induced} by $S$, i.e. $G[S]=(S,E_S)$
where $E_S:=\{xy\in E(G): x,y\in S\}$. Given another graph $H$ on the same vertex set as $G$,
we let $G-H:=(V,E(G)\setminus E(H))$. Given a path $P$ in a graph $G$, we let $\End(P)$ denote the set of its endpoints. For every family $\P$ of disjoint paths, we set $\End(\P):=\bigcup_{P\in\P} \End(P)$ and $E(\P):=\bigcup_{P\in\P} E(P)$. With $P_n,S_n,K_n$ and $K_{n,n}$ we denote the path on $n$ vertices,
the star on $n$ vertices (and $n-1$ edges), the complete graph on $n$ vertices
and the complete bipartite graph with vertex classes of size $n$ each, respectively. 
A cycle in a graph $G$ is called {\em Hamilton cycle} if it passes through every vertex of $G$;
in case such a cycle exists, $G$ is called {\em Hamiltonian}. A set of pairwise disjoint edges
in a graph $G$ is called a {\em matching}, and we call it a {\em perfect matching} if it covers
every vertex (but at most one, in case $v(G)$ is odd).\\

Assume an $(a:a)$ Maker-Breaker game is in progress. By $M$ we denote Maker's graph and by $B$
we denote Breaker's graph. If an edge is unclaimed by any of the players we call it \textit{free}.  
Each {\em round} (but maybe the very last one) consists of exactly one {\em move} of Breaker followed by a move of Maker, where the player claims $a$ edges. Claiming only one of these edges is called a {\em step} in the game. \\ %We will always assume that Breaker is the first player.\\

The rest of the paper is organized as follows. In Section~\ref{basic} we start with some preliminaries. In Section~\ref{wPM}, we prove Theorem~\ref{WeakPM}, in Section~\ref{wHCG} we prove Theorem~\ref{WeakHam}, in Section~\ref{sHCG} we prove Theorem~\ref{StrongHam} and in Section~\ref{wPkf} we prove Theorem~\ref{weakPk}. Finally, in Section~\ref{wSkf}, we prove Theorem~\ref{weakSk}.
 % Introduction and Results
\section{Preliminaries}\label{basic}

In the strategy for the $(a:a)$ Hamilton cycle game, given later in Section~\ref{wHCG},
Maker tries to maintain
a linear forest (a collection of vertex disjoint paths) as long as possible. 
We show that she can do so for the first $\lceil n/a \rceil -1$ rounds. 
Within the following (at most) two rounds, 
motivated by the rotation techniques of Pos\'a (see e.g.\ \cite{Posa}), she then adds further
edges to her graph and obtains a Hamilton cycle.
In order to guarantee that this is possible,
we prove the following statements.

\begin{proposition}\label{rotation}
Let $P_1$, $P_2$ be vertex disjoint paths in some graph $G,$
and $d_{\overline{G}}(v)\leq v(P_1)/2 - 1$ for every $v\in \bigcup_{i=1}^2 \End(P_i).$
Then there exists 
an edge $f\in E(P_1)$ and two edges $e_1,e_2\in E(G)$
such that $((E(P_1)\cup E(P_2))\setminus \{f\})\cup \{e_1,e_2\}$
induces a path $P$ of $G$ with 
$V(P)=V(P_1)\cup V(P_2)$ and $\End(P)\subseteq \End(P_1)\cup \End(P_2).$
\end{proposition}

\begin{proof} Let $\End(P_i)=\{x_i,y_i\}.$ 
For a vertex $v\in V(P_1)\setminus \{x_1\}$ we define its left-neighbour $v^-$
to be the unique neighbour of $v$ on the subpath from $x_1$ to $v.$
Similarly, for a vertex $v\in V(P_1)\setminus \{y_1\}$ we define its right-neighbour $v^+$
to be the unique neighbour of $v$ on the subpath from $v$ to $y_1.$
Let $S\subseteq V(P_1)$ be some subset, then we set
$S^-:=\{v^-:\ x_1\neq v\in S\}.$
Now, let $S_1:=N_G(x_1)\cap V(P_1)$ and $S_2:=N_G(x_2)\cap V(P_1).$
Then, by assumption $|S_1^-|\geq v(P_1)-1-d_{\overline{G}}(x_1)\geq v(P_1)/2$
and $|S_2| \geq v(P_1) - d_{\overline{G}}(x_2)> v(P_1)/2.$
Thus, $|S_1^-|+|S_2|>v(P_1)$ and therefore $S_1^-\cap S_2\neq \emptyset.$
Let $z\in S_1^-\cap S_2.$  Then $((E(P_1)\cup E(P_2))\setminus \{zz^+\})\cup \{x_1z^+,x_2z\}$
induces a path $P$ as claimed.
\end{proof}

\begin{corollary}\label{Hamiltonian}
Let $G$ be a graph on $n$ vertices; let $P_1,P_2,\ldots,P_t$
be pairwise disjoint paths in $G$ such that $\bigcup_{i=1}^t V(P_i)=V(G).$
Let $d_{\overline{G}}(v)\leq n/(2t) - 1$ for every $v\in \bigcup_{i=1}^t \End(P_i).$
Then there exists a set $E^*\subseteq E(G)$ of $2t$ edges
such that $\bigcup_{i=1}^t E(P_i) \cup E^*$ contains a Hamilton cycle of $G.$
\end{corollary}

\begin{proof}
W.l.o.g.\ let $v(P_1)=\max\{v(P_i):\ 1\leq i\leq t\}\geq n/t$ 
and therefore $d_{\overline{G}}(v)\leq v(P_1)/2 - 1$
for every $v\in \bigcup_{i=1}^t \End(P_i).$ 
Then, applying Proposition \ref{rotation} for $P_1$ and $P_2$,
we can find two edges $e_1,e_1'\in E(G)$ such that $E(P_1)\cup E(P_2)\cup \{e_1,e_e'\}$
contains a path $P_{12}$ with $V(P_{12})=V(P_1)\cup V(P_2)$
and $\End(P_{12})\subseteq \End(P_1)\cup \End(P_2).$
Now, set $P_1'=P_{12}, P_2'=P_3, P_3'=P_4, \ldots, P_{t-1}'=P_t.$
We can repeat this argument for $P_1'$ and $P_2',$ thus reducing the number of paths 
again by 1, using at most two further edges from $G.$ Doing this iteratively for $t-1$ 
iterations (using at most $2t-2$ edges)
we finally end up with a path $P$
such that $V(P)=\bigcup_{i=1}^t V(P_i)=V(G)$ 
and $\End(P)\subseteq \bigcup_{i=1}^t \End(P_i).$
Let $\End(P)=\{x,y\}.$ By assumption, $d_G(x),d_G(y)\geq n/2.$ 
Analogously to Bondy's proof of Dirac's Theorem for the existence of Hamilton cycles (see e.g.\ \cite{West}), we can find (at most) two edges $f_1,f_2$ in $G$ such that $E(P)\cup \{f_1,f_2\}$
contains a Hamilton cycle of $G$.
\end{proof}

 % Some useful Lemmas
\section{Weak Perfect Matching Game}
\label{wPM}
The main goal of this section is to prove Theorem \ref{WeakPM}.
However, we prove a slightly stronger result which roughly says that
Maker can claim a perfect matching rapidly
even if we play on a nearly complete bipartite graph.
This statement is used later in Section~\ref{wSkf}.

\begin{proposition}\label{WeakPMstronger}
Let $a,C\in \N$ be constants with $a\geq 2$, 
then for every large enough $n$ the following holds:
Let $G\supseteq K_{n,n}-H$ be a graph such that $e(H)\leq C$. 
Playing an $(a:a)$ Maker-Breaker game on $G$, Maker has a strategy
to gain a perfect matching of $G$ within at most $\lfloor n/a \rfloor +1$
rounds. In case $G=K_{2n}$, she can even win within
$\lceil n/a \rceil$ rounds.
\end{proposition}

Before we give the proof of this proposition, let us first see 
how Theorem~\ref{WeakPM} can be deduced.

\begin{proofof}{Theorem~\ref{WeakPM}}
For $a=1$ the proof is given in \cite{HKSS09}. So, let $a\geq 2.$
Assume first that $n=2k$ is even. 
Since a perfect matching of $K_n$ has $n/2$ edges, 
the game obviously lasts at least $\lceil n/(2a) \rceil$ rounds.
By Proposition~\ref{WeakPMstronger}, Maker has
a strategy to win within $\lceil k/a \rceil=\lceil n/(2a) \rceil$ 
rounds.
Assume then $n$ to be odd. According to the rules, Maker wins if she claims
a matching covering all but one vertex, i.e.\ a matching of size $(n-1)/2.$
This obviously takes at least 
\begin{align*}
\left\lceil \frac{n-1}{2a} \right\rceil =
\begin{cases}
\lceil \frac{n}{2a} \rceil, & \text{ if }2a\nmid n-1\\
\lceil \frac{n}{2a} \rceil -1, & \text{ if }2a\ | \ n-1\\
\end{cases}
\end{align*}
rounds. However, Maker for her strategy can consider to play on the graph 
$K_{n-1}\subseteq K_n$. By the previous argument this takes her at most 
$\lceil (n-1)/(2a) \rceil$ rounds, provided $n$ is large enough. 
\end{proofof}

\begin{proofof}{Proposition~\ref{WeakPMstronger}}
Assume that Breaker starts the game. 
Let $X\cup Y$ be the bipartition of $V(G)=V(K_{n,n})$.
Further assume that $H$ belongs to Breaker's graph. 

\medskip 

{\bf Maker's strategy} is split into two stages.

\medskip

{\bf Stage I.} The first stage lasts exactly $\lceil n/a \rceil -1$ rounds.
For $0\leq i\leq \lceil n/a \rceil -1$, Maker ensures that immediately 
after her $i^{\text{th}}$ move 
her graph consists of a matching $M_i\subseteq E(X,Y)$ of size $i\cdot a$ 
and a set of isolated vertices $I_i=V\setminus V(M_i)$
such that the following properties hold:
\begin{itemize}
\item[$(P1)$] $e_B(I_i)\leq \max\{C-ia,0\}$,
\item[$(P2)$] $\forall v\in I_i:\, d_B(v)<n/8.$
\end{itemize}
Maker chooses the edges $e_1=e_1^{(i)},\ldots, e_a=e_a^{(i)}$ 
of her $i^{\text{th}}$ move in the following way:\\
She sets $\delta=0$ if $i\leq \lceil C/a \rceil$, and $\delta=1$
otherwise. Then for every $1\leq j\leq a-\delta$ 
she sets $I_{i-1}^{(j)}=I_{i-1}\setminus V(e_1,\ldots,e_{j-1})$
and chooses $e_j=x_jy_j\in E(I_{i-1}^{(j)})\cap E(X,Y)$ such that 
\begin{itemize}
\item $x_j$ maximizes $d_B(z,I_{i-1}^{(j)})$ over all choices $z\in I_{i-1}^{(j)},$ and
\item $y_j$ maximizes $d_B(z,I_{i-1}^{(j)})$ over all choices $z\in I_{i-1}^{(j)}$ with $x_jz\in E(X,Y)\setminus E(B)$.
\end{itemize}
Afterwards, if $i> \lceil C/a \rceil$ (and so $\delta=1$), 
she sets $E(I_{i-1}^{(a)})=I_{i-1}\setminus V(e_1,\ldots,e_{a-1})$
and chooses an unclaimed edge $e_a=x_ay_a\in E(I_{i-1}^{(a)})\cap E(X,Y)$
in such a way that $x_a$ maximizes $d_B(z,V)$ over all choices $z\in I_{i-1}^{(a)}.$
Finally, she sets $I_i=I_{i-1}^{(a+1)}:=I_{i-1}\setminus V(e_1,\ldots,e_a)$.
\medskip

{\bf Stage II.} If $a\nmid n$ or $G=K_{2n}$, Maker plays one further round to complete a perfect matching
of $G$. Otherwise, she does so within two more rounds. The details of how she can do this are given later in the proof.

\medskip

It is evident that, if Maker can follow the strategy, 
she wins the game within the claimed number of rounds. 
Thus, it remains to show that Maker can follow the proposed strategy.

\medskip

{\bf Stage I.}
We prove, by induction on $i,$ 
that Maker can follow the strategy of Stage I and ensure the mentioned properties
to hold immediately after her $i^{\text{th}}$ move.
For $i=0$ there is nothing to prove.
So, let $i>0.$ Assume that $(P1)$ and $(P2)$ were true immediately after Maker's 
$(i-1)^{\text{st}}$ move for the matching $M_{i-1}$ and the set $I_{i-1}.$
To show that Maker can follow the strategy for round $i$, we inductively prove
the following claim.

\begin{claim}\label{WeakPM:claim_new}
For every $1\leq j\leq a-\delta$,
Maker can claim $e_j$ and ensure
that immediately after claiming that edge, 
$e_B(I_{i-1}^{(j+1)})\leq \max\{C-(i-2)a-2j,a-2j,0\}$ holds.
\end{claim} 

Indeed, if this claim is true, then Maker can claim the first $a-\delta$ edges, as described.
If $\delta=1$ and thus $i> \lceil C/a\rceil$, then  
after claiming the first $a-1$ edges,
we have 
$e_B(I_{i-1}^{(a)})=0$. That is,
Breaker has no edges among the remaining isolated vertices, and thus 
Maker can claim $e_a$ as described.

\begin{proofof}{Claim~\ref{WeakPM:claim_new}}
%\begin{proof} 
We apply induction on $j$, the number of steps in round $i$. At first, let $j=1$. If $i\leq \lceil C/a\rceil$, then by (P1) after 
Maker's $(i-1)^{\text{st}}$ move and
since Breaker has bias $a$, we have
$e_B(I_{i-1})\leq \max\{C-(i-2)a,a\}<n/2<|I_{i-1}\cap X|=|I_{i-1}\cap Y|$
right before Maker's first step in round $i$.
If otherwise $\lceil C/a\rceil<i\leq \lceil n/a \rceil -1$,
then analogously we have
$e_B(I_{i-1})\leq a<n-(i-1)a=|I_{i-1}\cap X|=|I_{i-1}\cap Y|$
right before Maker's first step in round $i$.
So, in either case, looking at Breaker's graph, 
none of the vertices
from $I_{i-1}\cap X$ can be adjacent 
to all vertices from $I_{i-1}\cap Y$,
and vice versa. So, Maker can claim $e_1$
as given by the strategy.
Now, if $e_B(I_{i-1}^{(1)})\geq 2,$ then by the choice of $e_1$
it easily follows that $d_B(x_1,I_{i-1}^{(1)})+ d_B(y_1,I_{i-1}^{(1)})\geq 2.$
But this gives $e_B(I_{i-1}^{(2)})\leq e_B(I_{i-1}^{(1)})-2\leq \max\{C-(i-2)a-2,a-2,0\}$.
Otherwise, if $e_B(I_{i-1}^{(1)})\leq 1,$ then $e_1$ is adjacent to all Breaker edges 
in $I_{i-1}^{(1)}$, ensuring $e_B(I_{i-1}^{(2)})=0.$

\medskip

Now, let $j>1$. After $e_{j-1}$ is claimed,  
$e_B(I_{i-1}^{(j)})\leq \max\{C-(i-2)a-2(j-1),a-2(j-1),0\}$ holds by induction.
If $i\leq \lceil C/a\rceil$, then
$e_B(I_{i-1}^{(j)})< n/2<|I_{i-1}^{(j)}\cap X|=|I_{i-1}^{(j)}\cap Y|.$
If $i> \lceil C/a\rceil,$
then 
$e_B(I_{i-1}^{(j)})\leq \max\{a-2(j-1),0\}< |I_{i-1}\cap Y|-(j-1)=|I_{i-1}^{(j)}\cap X|=|I_{i-1}^{(j)}\cap Y|.$
So, when Maker wants to claim her edge $e_j$, 
none of the vertices
from $I_{i-1}^{(j)}\cap X$ can be adjacent 
to all vertices from $I_{i-1}^{(j)}\cap Y$,
and vice versa.
So, as for the induction start, she can claim $e_j$ and ensure
$e_B(I_{i-1}^{(j+1)})\leq \max\{e_B(I_{i-1}^{(j)})-2,0\}\leq\max\{C-(i-2)a-2j,a-2j,0\}.$
\end{proofof}

When Maker claimed all the $a$ edges, she has a matching in $E(X,Y)$
of size $|M_i|=|M_{i-1}|+a=ia.$
 Thus, to finish the discussion of Stage I, it remains to show 
that the mentioned properties are maintained.
Using the claim, (P1) is given as follows:
If $i\leq \lceil n/a \rceil$, then immediately
after Maker's $i^{\text{th}}$ move,
$e_B(I_i)=e_B(I_{i-1}^{(a+1)})\leq \max\{C-(i-2)a-2a,a-2a,0\}=\max\{C-ia,0\}.$
If $i>\lceil n/a \rceil$, then
immediately after claiming $e_{a-1}$ we have
$e_B(I_{i-1}^{(a)})\leq \max\{C-(i-2)a-2(a-1),a-2(a-1),0\}=0.$ 
In particular, $e_B(I_i)=0$ follows then.
Assume now that (P2) is violated. Then after Maker's move there needs to be 
a vertex $v_i\in I_i$ of degree at least $n/8$ in Breaker's graph. 
In particular, $i\geq n/(8a).$ 
However, Maker then in each of the last twenty rounds 
chose (with her last edge)
an isolated vertex in her graph of maximum degree in Breaker's graph to be matched and therefore excluded from the set of isolated vertices.
That means, if after round $i$ there really was such a vertex $v_i,$
then twenty rounds before there must have been at least twenty vertices
of degree at least $n/8 - 20a$ in Breaker's graph, since otherwise $v_i$ would have been matched earlier.
But then, provided $n$ is large enough, Breaker has claimed more than $2n$ edges,
which is in contradiction to the number of rounds played so far.\\

{\bf Stage II.} We need to show that Maker can complete a perfect matching 
within one round or two rounds, respectively. From now on, 
let $t=|I_{\lceil n/a \rceil-1}\cap X|=|I_{\lceil n/a \rceil-1}\cap Y|$
be the number of isolated vertices in $X$ and $Y$, respectively, at the moment
when Maker enters Stage II.

\medskip

{\bf Case 1. $\mathbf{a\nmid n}$.} Observe that $t=n-|M_{\lceil n/a \rceil -1}|\leq a-1$ holds in this case.

{\bf Case 1.1.} Assume first that $t\leq a/2$.
Maker then partitions the set of isolated vertices
into $t$ pairs $\{v_1,w_1\},\ldots, \{v_t,w_t\}$
with $v_i\in Y$ and $w_i\in X$.
By Property $(P2)$, she then finds distinct edges 
$m_1=x_1y_1,$\ldots\ $m_t=x_ty_t$ in $M_{\lceil n/a \rceil-1}$
such that for each $1\leq i\leq  t,$ 
$x_i\in X$, $y_i\in Y$ and
the edges $v_ix_i$ and $w_iy_i$ do not belong to Breaker's graph.
She then claims the edges $v_ix_i$ and $w_iy_i$, in total $2t\leq a$.
By this, she creates a perfect matching of $G$.

{\bf Case 1.2.} Assume then that $a/2<t\leq a-1$.
Let $F$ be the graph induced by all free edges
between $I_{\lceil n/a \rceil-1}\cap X$ and $I_{\lceil n/a \rceil-1}\cap Y$.
Since by (P2) Breaker
can have at most $a$ edges among all isolated vertices,
then $e(F)\geq t^2-a$. Thus, the smallest vertex cover in $F$
is of size at least $(t^2-a)/t=t-a/t> t-2$. 
Therefore, by the theorem of K\"onig-Egev\'ary (see e.g.\ \cite{West}) $F$ has a matching of size at least $t-1$.
Maker claims this matching. 
This way, she creates a matching of $G$ of size $n-1$, and
two isolated vertices $v\in Y$ and $w\in X$.
Again, using Property (P1), she finds an edge
$xy$ in her matching, with $x\in X$ and $y\in Y$,
such that $vx$ and $wy$ are unclaimed. 
She claims these, and then she is done as before.
In total, she claims at most $(t-1)+2\leq a$ edges.

{\bf Case 2. $\mathbf{a|n}$.} Observe first that $t=a$. 
We now want to finish the perfect matching
within one round if $G=K_{2n}$. Otherwise, it is enough to
finish within two further rounds.

{\bf Case 2.1.} Assume that $G=K_{2n}$. 
When Maker enters Stage II, by Property (P1), Breaker claims at most $a$
edges in $G[I_{\lceil n/a \rceil-1}]$. So, there is a bipartite subgraph $G'\subset G[I_{\lceil n/a \rceil-1}]$
with classes of size $a$, such that Breaker claims
less than $a$ edges of $G'$. Set $F:=G'\setminus B$.
Then, analogously to Case 1.2, $F$ has a matching of size at least
$(t^2-(a-1))/{t}>a-1$. Thus, within one round, Maker can claim a matching of
size $a$ in $F$, which completes a perfect matching of $G$. 

{\bf Case 2.1.} Finally, assume $G\neq K_{2n}$. Let 
$F$ be the graph induced by all free edges
between $I_{\lceil n/a \rceil-1}\cap X$ and $I_{\lceil n/a \rceil-1}\cap Y.$
Since by (P2) Breaker
can have at most $a$ edges among all isolated vertices,
we analogously conclude that $F$ contains a matching of size at least
$a-1$. Maker claims such a matching in the first round of Stage II, 
and afterwards, 
she has a matching of $G$ of size $n-1$, and
two isolated vertices $v\in Y$ and $w\in X$.
If $vw$ is free, she claims it and wins. Otherwise, in the next round,
analogously to Case 1.2, she finds an edge
$xy$ in her matching, with $x\in X$ and $y\in Y$,
such that $vx$ and $wy$ are unclaimed. 
She then claims these two edges.
\end{proofof}

 % Fast Perfect Matching Game
\section{Weak Hamilton Cycle Game}
\label{wHCG}
\begin{proofof}{Theorem~\ref{WeakHam}}
For $a=1$ the proof is given in \cite{HS09}. So, let $a\geq 2$
and assume that Breaker starts the game.
Since a Hamilton cycle has $n$ edges, 
the game obviously lasts at least $\lceil n/a \rceil$ rounds.
Moreover, one easily verifies that, if $a=2$ and $n$ is even,
$\tau_{\ham}(a:a)\geq \lceil n/a \rceil +1.$ 
Indeed, assume in this case that 
Maker had a strategy to create a Hamilton cycle within $n/2$ rounds. 
Then, after her $(n/2 -1)^{\text{st}}$ round her graph would consist of 
two paths $P_1$ and $P_2$ (maybe one of length zero). 
In order to win in the next round, 
she would need to claim two edges between $\End(P_1)$ and $\End(P_2)$ 
in such a way that a Hamilton cycle is created. 
However, before this, Breaker can claim all edges of $E(x,\End(P_2))$ 
for some $x\in \End(P_1),$ therefore delaying Maker's win by at least one further round,
in contradiction to the assumption.

\medskip

Thus, it remains to prove that $\tau_{\ham}(a:a)\leq \lceil n/a \rceil + 1$
if $a=2$ and $n$ is even, and $\tau_{\ham}(a:a)\leq \lceil n/a \rceil $
otherwise (for large enough $n$ depending on $a$).

\medskip

%Before we describe Maker's strategy, we introduce some terminology. 
%Assume the game is in process. With $\P$ we denote the set of Maker's (maximal) paths, where an isolated vertex is seen as a path of length zero
%(empty path). Throughout the game this set is updated 
%meaning that, whenever Maker connects the endpoints of two paths $P_1,P_2\in \P$
%by an edge $e,$ we delete $P_1$ and $P_2$ from $\P$ and add the new path
%induced by $E(P_1)\cup E(P_2)\cup \{e\}$. 
%With $p_i$ we denote the size of $\P$ immediately after Maker's $i^{\text{th}}$ move.
%Observe that, as long as Maker's graph is a linear forest, $p_i=n-ia$ holds.
%  
%Now, let $e$ be some edge that is incident with the endpoints of two different paths 
%from $\P.$ We say that $e$ is {\em good}, if it is free;
%otherwise we call it {\em bad}.
%With $br_i$ we denote the number of bad edges right after
%Maker's $i^{\text{th}}$ move. If an edge $e$ is good, we set $D(e)$ 
%to be the number of bad edges adjacent to $e$. 
%Thus, $D(e)$ depends on the dynamic family $\P$ and on the considered round. 
%
%\medskip

{\bf Maker's strategy.}
The main idea of Maker's strategy is to create a linear forest, i.e.\ 
a graph which only consists of vertex disjoint paths. Her strategy is divided into three stages. In the first stage, she starts with a perfect matching,
similarly to the strategy given for $a=1$ in \cite{HKSS09}. 
Then, in the second stage, she connects the edges of the matchings to
create larger paths. Finally, in the third stage, when the number of paths is at most $a$, 
she completes a Hamilton cycle in at most two further rounds, making use of Proposition~\ref{rotation} and Corollary~\ref{Hamiltonian}. 

\medskip

Assume the game is in progress. By $\P$ we denote the set of Maker's (maximal) paths, where an isolated vertex is seen as a path of length zero
(empty path). Throughout the game this set is updated, 
meaning that whenever Maker connects the endpoints of two paths $P_1,P_2\in \P$
by an edge $e,$ we delete $P_1$ and $P_2$ from $\P$ and add the new path
induced by $E(P_1)\cup E(P_2)\cup \{e\}$. 
By $p_i$ we denote the size of $\P$ immediately after Maker's $i^{\text{th}}$ move.
Observe that, as long as Maker's graph is a linear forest, $p_i=n-ia$ holds.
  
Now, let $e$ be some edge that is incident with the endpoints of two different paths 
from $\P.$ We say that $e$ is {\em good}, if it is free;
otherwise we call it {\em bad}.
By $br_i$ we denote the number of bad edges right after
Maker's $i^{\text{th}}$ move. If an edge $e$ is good, we set $D(e)$ 
to be the number of bad edges adjacent to $e$. 
Thus, $D(e)$ depends on the dynamic family $\P$ and on the considered round. 

\medskip

{\bf Stage I.} Within $\lceil n/(2a) \rceil$ or $\lceil n/(2a) \rceil + 1$ rounds,
Maker claims a collection of vertex disjoint paths,
each of length at least 1, such that every vertex is incident with
one of these paths, and there is no further Maker's edge.
The details of how she can do this follow later in the proof. 
Afterwards, Maker proceeds with Stage II.

\medskip

{\bf Stage II.}
Let $t_1\in \{\lceil n/2a \rceil +1, \lceil n/2a \rceil +2\}$ be the round in which
Maker enters Stage II. Let $$t_2:=\lceil n/a \rceil - \lceil n/(6a^2) \rceil.$$
Maker now connects the paths of her collection $\P$.
To be able to do so, she needs to guarantee that the number of bad edges does not become too large. For that reason, if a lot of bad edges exist, 
she claims good edges that are adjacent to many bad edges (part IIa). 
Moreover, similarly to the perfect matching game, she maintains some
degree condition 
by caring about large degree vertices in Breaker's graph (part IIb). To be more precise:\\
For every $t_1\leq i\leq \lceil n/a \rceil -1,$ in her 
$i^{\text{th}}$, Maker claims edges $e_1=e_1^{(i)},\ldots ,e_a=e_a^{(i)}$ 
one after the other. The $j^{\text{th}}$ edge $e_j$ is claimed according to the following rules:

\begin{itemize} \item In order to choose $e_j$ she considers the following two cases.
\begin{itemize}
\item[{\bf IIa.}] If $i\leq t_2$ or $i \geq t_2 + 8,$ 
	Maker claims a good edge $e_j$ such that $D(e_j)$ is maximal.
\item[{\bf IIb.}] Otherwise, if $t_2+1\leq i\leq t_2+7,$ 
	she chooses a vertex $x_j\in \End(\P)$ of maximal  degree in Breaker's graph and then claims an arbitrary good edge $e_j=x_jy_j.$
\end{itemize}
	\item By claiming $e_j$, Maker connects two paths $P_{j,1},\ P_{j,2}\in\P.$	
	 Accordingly, she then deletes $P_{j,1},\ P_{j,2}$ from $\P,$ and adds to $\P$
		the path induced by $E(P_{j,1})\cup E(P_{j,2})\cup \{e_j\}$.  
		She updates the sets of good and bad edges 
		and the values $D(\cdot)$ before she proceeds with $e_{j+1}.$
\end{itemize}

{\bf Stage III.} 
If $a=2$ and $n$ is even, Maker claims a Hamilton cycle within the next two rounds.
Otherwise, she does so within one round. The details of how Maker can do this, follow later in the proof.

\medskip

It is evident that, if Maker can follow the strategy, 
she wins the game within the desired number of rounds. 
Thus, it remains to show that Maker can follow the proposed strategy.
Before that, let us prove the following propositions
which bound the number of bad edges throughout the game.

\medskip

\begin{proposition}\label{WeakHam:proposition0}
At any point of the game, when $|\P|\geq 2$ and
\begin{itemize}
\item each $v\in \End(\P)$ is incident with at least one good edge,
\item $D(e)\leq 1$ holds for all good edges $e$.
\end{itemize}
Then the number of bad edges is at most $1.$
\end{proposition}

%\begin{proofof}{Proposition \ref{WeakHam:proposition0}}
\begin{proof}
By assumption, each vertex in $\End(\P)$ is incident with at most one bad edge.
If there were at least 2 bad edges, 
they would form a matching. Maker could then find a good edge $e'$
that is adjacent to two of these edges, in contradiction to $D(e')\leq 1.$
\end{proof}

\begin{proposition}\label{WeakHam:proposition1}
At any point of the game, when $|\P|\geq 4$ and
\begin{itemize}
\item each $v\in \End(\P)$ is incident with at least one good edge,
\item $D(e)\leq 2$ holds for all good edges $e$.
\end{itemize}
Then the number of bad edges is at most $|\P|.$
\end{proposition}

%\begin{proofof}{Proposition \ref{WeakHam:proposition1}}
\begin{proof}
By assumption, each vertex in $\End(\P)$ is incident with at most two bad edges.
If each vertex is incident with at most one bad edge,
then the bad edges form a matching on $\End(\P)$.
Thus, there can be at most $|\P|$ such edges. So, assume that there is a vertex $x$ incident with exactly two bad edges $xy_1$ and $xy_2$. Then the number of bad edges
is at most $4$, which can be seen as follows: 
Let $z$ be the other endpoint of the path that
$x$ belongs to. If $y_1$ and $y_2$ belong to the same path in $\P$, then the only further edges that could be bad are $zy_1$ and $zy_2.$ Indeed, if there was another endpoint $w\notin \{x,y_1,y_2,z\}$ 
incident with some bad edge, then $D(xw)\geq 3,$ a contradiction.
Otherwise, if $y_1$ and $y_2$ belong to different paths, then 
similarly one observes that $y_1y_2$ is the only edge that could be bad
besides $xy_1$ and $xy_2$. 
\end{proof}

\begin{proposition}\label{WeakHam:proposition2}
Let $br$ and $p=|\P|>2$ be the numbers of bad edges and Maker's paths, respectively,
immediately before Maker claims some edge $e$. If
$br<2p-2,$ then the following holds:

\begin{enumerate}[i)]
\item Each vertex in $\End(\P)$ is
incident with at least one good edge. %(In particular, Maker can
% easily claim an edge as asked in Stage II.)
\item If $e$ is a good edge such that $D(e)$ is maximal 
(at the moment when $e$ is chosen),
and if $br'$ and $p'$ are the numbers of bad edges and Maker's paths
immediately after Maker claimed $e,$
then again $br'<2p'-2.$ Moreover, $br'=0$ if $D(e)\leq 1.$ 
\end{enumerate}
\end{proposition}

%\begin{proofof}{Proposition \ref{WeakHam:proposition2}}
\begin{proof}
Part $i)$ of the proposition holds,
since for each vertex the number
of incident good edges and the number of incident bad edges
sums up to $2(p-1).$\\
\noindent So, consider part $ii)$.
If $D(e)\geq 2,$ then Maker gets rid of at least two bad edges by claiming $e$
which gives $br'\leq br- 2<2p-4=2p'-2.$ 
Otherwise, if $D(e)\leq 1,$ then $br\leq 1$,
by Proposition~\ref{WeakHam:proposition0}. 
By the choice of $e$, it follows that $br'=0.$
\end{proof}

\medskip

The last proposition turns out to be very helpful for the
discussion of Stage II. The reason is that whenever $br<2p-2$
holds, then it tells us that Maker can claim a good edge as asked by
the strategy of Stage IIa. Moreover,
after Maker claimed such an edge (and thus $br'<2p'-2$ holds),
we can reapply this proposition, and continue this way
until Maker's move in Stage IIa is over.
With all the previous propositions in hand, let us now prove that Maker can follow the strategy.

\medskip

{\bf Stage I.} If $n$ is even, Maker plays $\lceil n/(2a) \rceil$ rounds
according to the strategy given for the perfect matching game.
This produces a perfect matching (thus every vertex is covered)
plus at most $a$ further edges that, together with the matching edges, 
form a linear forest. In case this strategy stops in round
$\lceil n/(2a) \rceil$ before Maker claimed exactly $a$ edges,
Maker claims further edges that maintain a linear forest.
Note that this is possible since $n$ is large enough and
Breaker so far claimed at most $n/2$ edges.
If $n$ is odd, Maker plays $\lceil n/(2a) \rceil$ rounds
on $K_{n-1}\subseteq K_n$,
analogously occupying a family of paths of length at least 1, 
covering all vertices of $K_{n-1}$. 
In the next round, she connects the unique vertex $v\in V(K_n)\setminus V(K_{n-1})$ 
to an endpoint of one of her paths,
and afterwards claims $a-1$ further edges such that her graph remains a linear forest. 
Again, this is possible, since Breaker so far claimed at most $n/2+2a$ edges.
\medskip

{\bf Stage II.}
Observe first that, when Maker enters Stage II, 
her collection $\P$ consists of $p_{t_1-1}=n-(t_1-1)a\geq n/2 - 2a$ paths.
Moreover, immediately after her previous move the number of bad edges was
$br_{t_1-1}\leq (t_1-1) a\leq n/2 + 3a\leq p_{t_1-1}+5a.$
The following claim splits Stage II naturally into three parts
and ensures for each part that Maker can follow the proposed strategy.

\medskip

\begin{claim}\label{WeakHam:claim}
For Stage II the following is true.
\begin{enumerate}[(a)]
\item For every $t_1\leq i\leq t_2,$ 
Maker can make her $i^{\text{th}}$ move according to Stage IIa.
After that move, 
\begin{align}\label{inequalities1}
br_i\leq p_i-a \quad \text{ or } \quad br_i-p_i<br_{i-1}-p_{i-1}.
\end{align}
In particular, $br_i\leq p_i+5a$ for every $t_1\leq i\leq t_2$,
and $br_{i}\leq p_{i}-a$  for every $t_1+6a\leq i \leq t_2.$
\item
For every $t_2+1\leq i\leq t_2+7,$ 
Maker can follow her $i^{\text{th}}$ move according to Stage IIb,
and after that $br_i\leq p_i+13a.$
\item
For every $t_2+8\leq i\leq \lceil n/a \rceil -1,$ 
Maker can follow her $i^{\text{th}}$ move according to Stage IIa.
After that move, \\[-0.7cm]
\begin{align}\label{inequalities2}
br_i\leq \max\{p_i-a,0\} \quad \text{ or } \quad br_i-p_i<br_{i-1}-p_{i-1}.
\end{align}
In particular, $br_i\leq p_i+13a$ always, and
$br_i\leq \max\{p_i-a,0\}$ for every $i\geq t_2+14a+8.$
\end{enumerate}
\end{claim}

{\bf Proof}
%\begin{proof}
\begin{enumerate}[(a)]
\item 
We prove the statement by induction on $i.$
When Maker has to make her $i^{\text{th}}$ move, by induction hypothesis,
she sees at most $br_{i-1}+a\leq p_{i-1}+6a<2p_{i-1}-2$
bad edges on the board. Thus, by Proposition~\ref{WeakHam:proposition2},
Maker can follow her strategy and claim $e_1,\ldots, e_a.$
Now, let $D_j$ denote the value of $D(e_j)$ 
at the moment when $e_j$ is chosen, and observe that
$D_1\geq D_2\geq \ldots \geq D_a.$\\
\textbf{Case 1.} If $D_a\leq 1$, then by Proposition~\ref{WeakHam:proposition2} $ii)$,
we obtain $br_i=0.$\\
\textbf{Case 2.} If $D_1\geq 3$ and $D_a\geq 2,$ then
Breaker in his $i^{\text{th}}$ move created at most $a$ bad edges,
while Maker gets rid of $\sum_{j=1}^a D_j\geq 3+2(a-1)$ bad edges. We conclude that
$br_i-p_i\leq br_{i-1}+a-(3+2(a-1))-(p_{i-1}-a)<br_{i-1}-p_{i-1}.$\\
\textbf{Case 3.} If $D_1=D_a=2,$ then after Breaker's 
$i^{\text{th}}$ move there were at most $p_{i-1}$
bad edges, as given by Proposition~\ref{WeakHam:proposition1}.
Maker in her $i^{\text{th}}$ move decreases the number of bad edges by 
$\sum_{i=1}^a D_i=2a,$ 
while the number of paths only decreases by $a$. This gives $br_i\leq p_i-a$.

Thus, in either case (\ref{inequalities1}) holds. Finally, it follows that
$br_i\leq p_i+5a$ for all $i$ and 
$br_{i}\leq p_{i}-a$ for all $i\geq t_1+6a$, since
the difference $br_i-p_i$ decreases as long as it is larger than $-a.$

\item 
If Maker can follow the strategy, then one verifies that
$br_i\leq br_{t_2}+7a\leq p_{t_2}+6a=p_{t_2+7}+13a\leq p_i+13a$
for every $t_2+1\leq i\leq t_2+7.$
On the other hand, this inequality ensures that, when
Maker has to make her $i^{\text{th}}$ move, 
each vertex in $\End(\P)$ is incident with at least 
$2(p_{i-1}-1)-(p_{i-1}+13a) \geq (n-(t_2+6)a)-14a \geq n/(7a)$
good edges. Therefore, Maker can follow the proposed strategy for Stage IIb.

\item 
Similarly to the proof of (a) we apply induction on $i$.
Assume the statement was true until round $i-1.$
If $i< t_2+8+14a,$ then $br_{i-1}+a\leq p_{i-1}+14a < 2p_{i-1}-2.$
If $i\geq t_2+8+14a,$ then $p_{i-1}>a$ and,
since by (\ref{inequalities2}) the difference $br_i-p_i$ decreases 
as long as it is larger than $-a,$ 
we obtain $br_{i-1}+a\leq p_{i-1}<2p_{i-1}-2.$
So, in any case, when Maker starts her $i^{\text{th}}$ move,
she sees at most $br_{i-1}+a<2p_{i-1}-2$ bad edges on the board.
Thus, applying Proposition~\ref{WeakHam:proposition2} we know that 
Maker can follow the strategy for the current move. 
The proof of (\ref{inequalities2}) is done similarly to the proof of 
(\ref{inequalities1}) in (a). Indeed, Case 1 and 2 from that proof 
are handled analogously. Case 3 can be done as before, as long as
Proposition~\ref{WeakHam:proposition1} applies, i.e.\ as long as $p_{i-1}\geq 4.$
Since $p_{i-1}>a,$
the only time when this does not happen
is when $a=2$ and $p_{i-1}=p_{\lceil n/a \rceil - 2}=3.$
But then $i\geq t_2+8+14a$ and thus the number of
bad edges is at most $br_{i-1}+a\leq 3=p_{i-1},$ 
which is enough to handle Case 3 analogously.
Finally, $br_i\leq \max\{p_i-a,0\}$ for $i\geq t_2+14a+8$ holds,
since by (\ref{inequalities2}) the difference $br_i-p_i$ 
decreases as long as it is larger than $-a.$ {\hfill $\Box$}
\end{enumerate}
% \end{proof}
\medskip

{\bf Stage III.} Let $p$ be the size of $\P$ when Maker enters Stage III, 
and observe $p=p_{\lceil n/a \rceil -1}\leq a.$
In order to create a Hamilton cycle within 1 or 2 further rounds, we now 
make use of Proposition~\ref{rotation} and Corollary~\ref{Hamiltonian}. 
Before doing that, we need the following claim.

\begin{claim}\label{WeakHam:claim2}
Right before Maker's first move in Stage III, the following properties hold:
\begin{itemize}
\item[(H1)] The number of bad edges is at most $a,$
\item[(H2)] $\forall\ v\in \End(\P):\ d_B(v)< n/(3a).$
\end{itemize}
\end{claim}

%\begin{proofof}{Claim \ref{WeakHam:claim2}}
\begin{proof}
By Claim~\ref{WeakHam:claim} (c) we have 
$br_{\lceil n/a \rceil -1}\leq \max\{p_{\lceil n/a \rceil -1}-a,0\}=0.$
Breaker in his $\lceil n/a \rceil ^{\text{th}}$ move creates
at most $a$ bad edges, proving (H1).
Assume now that (H2) does not hold, i.e.\
there is a vertex $v\in \End(\P)$ with degree at least $n/(3a)$ in Breaker's graph,
right after Breaker's $\lceil n/a \rceil ^{\text{th}}$ move.
Then, in round $t_2+1,$ the degree of vertex $v$ in Breaker's graph is 
at least $n/(3a)- a(\lceil n/a \rceil - t_2)\geq n/(6a)-a.$
Now, Maker did not claim a good edge incident to $v$
so far. Thus, whenever Maker claimed an edge in Stage IIb, one of its endpoints already had degree at least $n/(6a)-a$ in Breaker's graph.
But, since Maker claims $7a$ independent edges throughout Stage IIb,
this means that Breaker needs to have at least $7a$ vertices in his graph of degree at least $n/(6a)-a,$ which gives the existence of more than $n$ Breaker's edges, 
in contradiction to the number of rounds played so far.
\end{proof}

\medskip

Finally, we show how Maker completes her Hamilton cycle by case distinction on $p.$\\

{\bf Case $\mathbf{p\leq a/2}$.} Applying Corollary~\ref{Hamiltonian} 
(with $G=K_n\setminus B$; using Claim~\ref{WeakHam:claim2} (H2)), 
Maker can find (at most) $2p\leq a$ free edges to finish a Hamilton cycle.
Maker claims these and is done.

\medskip

{\bf Case $\mathbf{p=(a+1)/2}$.} Observe that $a\geq 3$ and therefore,
by Claim~\ref{WeakHam:claim2} (H1),
the number of good edges is at least 
$4\binom{p}{2}-a>0.$
Maker at first claims
one such good edge, thus reducing the number of paths to $p-1.$
Afterwards, applying Corollary~\ref{Hamiltonian} as before,
she can find $2(p-1)$ free edges finishing a Hamilton cycle.
She claims these edges, which is possible since $2(p-1)+1=a.$

\medskip

{\bf Case $\mathbf{p=(a+2)/2}$ and $\mathbf{a > 4}$.}  
By Claim~\ref{WeakHam:claim2} (H1),
the number of good edges is at least 
$4\binom{p}{2}-a\geq 2a+2.$
That is why we can find at least two good edges
such that claiming them keeps Maker's graph being a linear forest.
%(Indeed, if Maker claims one good edge $e=xy,$
%the only edges that are not allowed to be chosen as the second edge
%are those incident with $x$ or $y$, and those between
%the endpoints of those paths that $x$ and $y$ belong to,
%in total at most $4p-5=2a-1.$) 
Maker at first claims
these two good edges, thus reducing the number of paths to $p-2.$
Afterwards, applying Corollary~\ref{Hamiltonian} as before,
she can find $2(p-2)$ free edges finishing a Hamilton cycle.
She claims these edges, which is possible since $2(p-2)+2=a.$

\medskip

\textbf{Case $\mathbf{p=3}$ and $\mathbf{a=4}$.} Similarly to the previous case, by Claim~\ref{WeakHam:claim2} (H1),
the number of good edges is at least 
$4\binom{p}{2}-a= 8.$
It is easily checked that we can find two good edges such that claiming them Maker's graph is a Hamilton path. Then, applying Corollary~\ref{Hamiltonian} and Claim~\ref{WeakHam:claim2} (H2), she can close this path into a Hamilton cycle by claiming at most two further edges. 

\medskip

{\bf Case $\mathbf{p=2}$ and $\mathbf{a=2}$.} In this case, $n$ is even,
and we are allowed to play two further rounds. When Maker enters Stage III,
her graph consists of two paths $P_1$ and $P_2.$ Applying Proposition~\ref{rotation} 
(with $G=K_n\setminus B$; using Claim~\ref{WeakHam:claim2} (H2)), 
she can claim two edges to obtain
a path $P$ covering $V(P_1)\cup V(P_2)=V$ in the first round. 
Then similarly, applying Corollary~\ref{Hamiltonian},
she can finish a Hamilton cycle in the next round.

\medskip

{\bf Case $\mathbf{p\geq (a+3)/2}$.}  
Observe that $a\geq p\geq 3$ and that, when Maker enters Stage III,
the number of bad edges is at most $a<2p-2.$
Thus, by Proposition~\ref{WeakHam:proposition2},
Maker at first can claim $p-2$ edges as in Stage IIa.
Afterwards, her graph consists of exactly two paths ($|\P|=2$), while, by the same proposition, 
the number of bad edges is smaller than $2|\P|-2=2.$ Thus, one finds two good edges that finish a Hamilton cycle. Maker claims these, which is possible as $(p-2)+2\leq a.$
\end{proofof}

%\newpage

\section{Strong Hamilton Cycle Game}
\label{sHCG}
\begin{proofof}{Theorem~\ref{StrongHam}}

At first we give a short description of a strategy for Red, 
and then we show that Red indeed can follow that strategy
and win the $(2:2)$ Hamilton cycle game in the desired number of rounds. 
As in the proof for the corresponding weak game,
Red starts by maintaining a linear forest for all but a small constant number of rounds. 
Then she completes a Hamilton cycle in her graph, while blocking possible Hamilton cycles
in Blue's graph.

Let ${\cal S}_{\ham}$ be Maker's strategy given in the previous chapter for the
$(2:2)$ Weak Hamilton cycle game. Assume that Red's graph is a collection $\P$
of paths. Again, an edge $e$ between the endpoints of different paths from $\P$ is called {\em good} 
if it is unclaimed. Otherwise, it is called {\em bad}.
For a good edge $e$ we set $D(e)$ 
to be the number of bad edges adjacent to $e$.
Red's strategy is divided into the following three stages.

\medskip

{\bf Stage I.} For the first $n/2 - 2$ rounds, Red follows the strategy ${\cal S}_{\ham}$,
thus creating a collection of 4 non-empty paths covering all vertices of $K_n.$

\medskip

{\bf Stage II.} At the very beginning of Stage II,
let $I_2$ denote the set of isolated vertices in the graph of Blue,
and let $\P_2$ be the collection of Red's paths (from Stage I). 
In round $n/2 - 1$, Red claims two good edges $e_1$ and $e_2$ 
such that the following properties hold:
\begin{itemize}
\item[(S1)] $E(\P_2)\cup \{e_1\}\cup \{e_2\}$ induces a collection $\P_3$
of two non-empty paths. Moreover, immediately after Red's move in round $n/2-1$,
there is no bad edge among the vertices of $\End(\P_3)$.
\item[(S2)] If $\End(\P_2)\cap I_2\neq \emptyset$, then $(e_1\cup e_2)\cap I_2\neq \emptyset.$
(That is, Red decreases the number of Blue's isolated vertices among the
endpoints of her paths, if this number is not zero.) 
\end{itemize}
The details of how she can claim her edges follow later in the proof.

\medskip

{\bf Stage III.} Within at most 2 further rounds, Red creates a Hamilton cycle.
Moreover, in the meantime she prevents the same in Blue's graph. The details of how she can do this follow later.

\medskip 

It is evident that, if Red can follow the strategy, 
she wins the game within the desired number of rounds. 
Thus, it remains to show that Red indeed can follow the proposed strategy.

\medskip

{\bf Stage I.} We already saw that Red/Maker can follow the strategy ${\cal S}_{\ham}$. So, Red creates a collection $\P_2$ of non-empty paths and, since she claims
$2\cdot (n/2 - 2)=n-4$ edges in total, we get $|\P_2|=4$.

\medskip
 
{\bf Stage II.} Recall that for the strategy ${\cal S}_{\ham}$,
the parameters $p_i$ and $br_i$ were introduced to denote the number
of Maker's/Red's paths and the number of bad edges immediately
after Maker's/Red's $i^{\text{th}}$ move, respectively. 
By Claim~\ref{WeakHam:claim} (c) we then have that 
immediately after Red's last move in Stage I,
the number of bad edges is at most
$br_{n/2 - 2}\leq \max\{p_{n/2 - 2}-2,0\}=2.$ Thus, when Red enters Stage II
there can be at most $2+2=4$ bad edges. We distinguish between the three cases.

\medskip

{\bf Case 1.} The number of bad edges is at most $3$.
Then Red at first chooses a good edge $e_1$ such that $D(e_1)$
is maximal (w.r.t.\ $\P_2$) and creates a collection
$\P_2'$ induced by $E(\P_2)\cup \{e_1\}.$ 
Note that, at the moment when 
$\P_2'$ is created, the number of bad edges is at most 1.
Indeed, if $D(e_1)\geq 2$, then at least two bad edges disappear.
Otherwise, if $D(e_1)\leq 1$, then by
Proposition~\ref{WeakHam:proposition0}, the number of bad edges was already at most 1.  
So, after $e_1$ is claimed, we have exactly three paths 
and at most one bad edge.
Then, Red chooses an arbitrary edge $e_2$ which is good
(w.r.t.\ the new collection $\P_2'$)
such that the following holds:
It is adjacent to the remaining bad edge if one exists;
and it is incident with some vertex from $I_2$ if $I_2\cap \End(\P_2')\neq \emptyset$.
It is easy to check that Red indeed can do so,
and properties (S1) and (S2) hold then.

\medskip

{\bf Case 2.} The number of bad edges is $4$, and there
is some good edge $e$ with $D(e)\geq 3$ (w.r.t.\ $\P_2$).
Then Red chooses $e_1$ and $e_2$ as in Case 1.
Just note that, when $e_1$ is chosen, the number of bad edges 
drops to at most 1. The rest follows analogously to Case 1.

\medskip

{\bf Case 3.} The number of bad edges is $4$,
and we have $D(e)\leq 2$ for every good edge $e$.
Then the subgraph of Blue's graph induced on $\End(\P_2)$ is either a matching, or it consists of a 4-cycle and 4 isolated vertices.
Indeed, if we don't have a matching, then there needs to be 
some $x\in \End(\P_2)$ which is incident with exactly
two bad edges $xy_1$ and $xy_2$. Now, as in the proof
of Proposition~\ref{WeakHam:proposition1}
either there is some endpoint $z\in \End(\P_2)$
such that $zy_1$ and $zy_2$ can be the only further 
bad edges, or $y_1y_2$ is the only further edge that can be bad.
The second case cannot happen, since we have 4 bad edges,
and so there needs to be a 4-cycle (with vertices $x,y_1,y_2,z$).

Now, in either case, it is easy to see that
Red can choose a good edge $e_1$ such that the following holds:
$e_1$ is adjacent to
exactly two bad edges; and $e_1$ is incident
with some vertex from $I_2$ if $I_2\cap \End(\P_2)\neq \emptyset.$
Afterwards,
the collection $\P_2'$ induced by $E(\P_2)\cup \{e_1\}$ 
consists of three paths, while the number of bad edges is 2. 
Red then chooses $e_2$ to be good w.r.t.\ $\P_2'$,
in such a way that $e_2$ intersects both of the remaining bad edges.
Again, Red can easily do so, and by this ensure the properties (S1) and (S2) hold.

\medskip

{\bf Stage III.}
When Red enters Stage III, her graph is
the collection $\P_3$ with properties (S1) and (S2). 
Moreover, Property (H2) from Claim~\ref{WeakHam:claim2} holds again:
If there was a vertex of degree at least $n/(3a)$ before Red's
first move of Stage III (right after Blue's $(n/2-1)^{\text{st}}$ move),
then analogously to the proof of
Claim~\ref{WeakHam:claim2} in one of the previous rounds Blue must have had at least $7a$ vertices of degree at least $n/(6a)-a$ in his graph, a contradiction.
Now, in the following we describe how 
Red finishes her Hamilton cycle while preventing such a cycle in Blue's graph.

For this, let $B_2$ denote Blue's graph right at the beginning
of Stage II (i.e.\ after Blue's $(n/2-2)^{\text{nd}}$ move), 
and let $B_3$ be his graph at the beginning of Stage III.

\medskip

{\bf Case 1.} Assume that $B_3$ 
satisfies one of the following three properties:
\begin{itemize}
\item $B_3$ contains a cycle.
\item $B_3$ has a vertex of degree at least 3.
\item $B_3$ has at least 3 components.
\end{itemize}
Then, since $|E(B_3)|=n-2$, one needs to add at least
3 edges to $B_3$ in order to get a Hamilton cycle.
That is, Blue cannot finish a winning set
before round $n/2+1$. Therefore, Red wins, if she can finish
a Hamilton cycle within the next two rounds (round $n/2$ and $n/2+1$). 
For this, just observe that 
with Property (H2) from Claim~\ref{WeakHam:claim2} in hand, Red can just follow Stage III of Maker's strategy from
the $(2:2)$ Weak Hamilton cycle game.   

\medskip

{\bf Case 2.} Assume that $B_3$ 
satisfies none of the three properties given in the first case.
Then, $B_3$ is a collection of exactly 2 (maybe one-vertex) paths
that cover all vertices. Therefore, $B_2$ is a collection of
4 (maybe one-vertex) paths and thus has at most 3 isolated vertices
(i.e.\ $|I_2|\leq 3$).

Now, if at the beginning of Stage III there are two good edges
in $\End(\P_3)$ that finish a Hamilton cycle in Red's graph,
then Red just claims these and wins the game.

So, we can assume that there is a vertex $w\in \End(\P_3)$,
which is incident with two bad edges $f_1=wy_1$ and $f_2=wy_2$
before Red's $(n/2)^{\text{th}}$ move, $y_1,y_2\in \End(\P_3)$.
Since, by Property (S1), there was no bad edge among $\End(\P_3)$
right after Red's $(n/2-1)^{\text{st}}$ move, we have that
$f_1$ and $f_2$ were claimed in round $n/2-1$ (i.e.
$f_1,f_2\in B_3\setminus B_2$). But then, since $B_3$
does not have a vertex of degree at least 3, we know that
$w$ must be isolated in $B_2$ (i.e.\ $w\in I_2$). 
Thus $|I_2|\geq 2$, 
since otherwise in Stage II (Property (S2)) we would ensure that
$I_2\cap \End(\P_3)=\emptyset$, in contradiction to the existence
of $w$. In particular, there is a vertex $x\in I_2$ with which
Red claims an incident edge in Stage II, and thus $x\notin \{y_1,y_2\}$.
As $|I_2|\leq 3$, it follows that
at least one of the vertices $y_1,y_2$ does not belong to $I_2$,
w.l.o.g.\ let $y_1\notin I_2$. We are left with two cases.

\medskip

{\em Case 2.1.} $|I_2|=3$. Then $B_2$ consists
of 3 isolated vertices (including $w$) and a path $P_{B_2}$ with $n-3$ vertices.
Then, the vertex $y_1\notin I_2$ must be the endpoint of the path $P_{B_2}$, 
since $E(P_{B_2})\cup f_1\subseteq E(B_3)$ does not give a vertex of degree at least 3.
Since $f_2=wy_2\in E(B_3)$ cannot create a cycle,
we have $y_2\in I_2$. But then, $B_3$ 
consists of one isolated vertex $z$, and a path $P_{B_3}$
with $n-1$ vertices. Red's strategy is as follows: 
In the first move, Red takes one edge which is good (w.r.t.\ $\P_3$) 
and creates a Hamilton path. This is possible, since by (S1)
there are only two bad edges between the two paths of $\P_3$.
For the second edge, she chooses one edge between $z$ 
and $\End(P_{B_3})$. In his next move, Blue cannot close a Hamilton cycle,
since this would need the two edges between $z$ and $\End(P_{B_3})$. 
In round $n/2+1$, using Claim~\ref{WeakHam:claim2} (H2) and Corollary~
\ref{Hamiltonian}, 
Red then completes a Hamilton cycle.

\medskip

{\em Case 2.2} $|I_2|=2$. Then $B_2$ consists
of 2 isolated vertices (including $w$) and two (non-empty) paths 
covering the remaining $n-2$ vertices.
By Stage II (see Property (S2)), 
we ensured that $|I_2\cap \End(\P_3)|\leq 1$,
and so $w$ is the unique vertex in $I_2\cap \End(\P_3)$.
In particular, $y_1,y_2\notin I_2$. Since $B_3$ has no cycles and no vertex of degree
at least 3, and since $f_1,f_2\in B_3$,
the vertices $y_1$ and $y_2$ must be the endpoints of different paths from $B_2$.
But then, $B_3$ again has exactly one isolated vertex and one path
with $n-1$ vertices. So, we can proceed as in Case 2.1.
\end{proofof}
 % Fast Hamilton Cycle Game
\section{$P_k$-factor game}
\label{wPkf}

\begin{proofof}{Theorem~\ref{weakPk}}
If $k=2$ and $a$ is any constant, then $\pkf$ is a perfect matching and we can use the proof of Theorem~\ref{WeakPM}. So, we let $k\geq 3$ and we fix some $\delta < 1/(8k)$. 
We first give a Maker's strategy. Then we prove that she can follow it and win within $\left\lceil(k-1)n/(ka)\right\rceil$ rounds.\\

\noindent \textbf{Maker's strategy} is to build $n/k$ vertex disjoint paths of length $k-1$. During the course of the game, the collection of all paths in her graph is denoted by $\P$. Each path in $\P$ belongs to exactly one of the three classes: 
$\P_u$, which denotes the collection of \textit{unfinished} paths (the paths of length at most $k-3$)%, and are denoted by $\P_u$. 
, $\P_f$, which denotes the collection of the \textit{finished} paths (the paths of length exactly $k-2$) 
 or $\P_c$, which denotes the collection of \textit{complete} paths %($\P_c$) is a set of 
(the paths of length exactly $k-1$). Maker's strategy consists of three stages. In Stage I of her strategy, Maker makes sure that every unfinished path becomes (at least) a finished path, while in the following stages she aims for complete paths.
The set of isolated vertices in Maker's graph is denoted by $U=V\setminus V(\P)$. By $\End(\P)$ we denote the set of endpoints of all paths. At the beginning, $\P:=\P_u$ contains $n/k$ arbitrarily chosen vertices; $\P_f$ and $\P_c$ are empty. If $P$ is a path in Maker's graph, then $v_1^P$ and $v_2^P$ represent its endpoints. \\

%%%%%%%%%%%%%%%%%%%%%%
% SHORT STRATEGY
%%%%%%%%%%%%%%%%%%%%%%

\noindent \textbf{Stage I.} In this stage, Maker plays as follows: She gradually extends the unfinished paths with the vertices from $U$ until they are finished. From time to time,
we allow her to complete some of these paths in order to keep control on the distribution of
Breaker's edges (as described by properties (Q1)--(Q3) in the following paragraph). After each step, the sets $\P_u, \P_f, \P_c$ and $U$ are dynamically updated in the obvious way. That is, whenever Maker extends one of her paths, $P$, by some vertex $u\in U$, this vertex is removed from $U$ and added to $P$, while $P$ may be moved from $\P_u$ to $\P_f$ or from $\P_f$ to $\P_c$ according to its new length. \\

During Stage I, for a given graph $G$, we say that $(G,\P)$ is {\em good} if the following properties hold:
\begin{itemize}
\item[$(Q1)$] $\forall u \in U : d_G(u,\End(\P_u\cup \P_f))<\delta n$;
\item[$(Q2)$] $G[U] = \emptyset$;
\item[$(Q3)$] $\forall P \in \P_u \cup \P_f: d_G(v_1^P,U)+d_G(v_2^P,U)\leq 1$.
\end{itemize}

In each move during this stage, Maker claims $a$ free edges between $U$ and $\End(\P_u \cup \P_f)$, so that after her move $(B,\P)$ is good.

In her $i^{th}$ move, Maker chooses the edges $e_1=e_1^{(i)},\dots,e_a=e_a^{(i)}$ one after another. She makes sure that for every $t\in \{0,1,\ldots,a\}$
the following holds:

\begin{itemize}
\item[$(Q4)$] Immediately after the edges $e_1,\ldots,e_t$ are claimed (and the paths are updated accordingly), there is a subgraph $H=H_t\subseteq B$ with $e(H)=a-t$ such $(B\setminus H,\P)$ is good.
\end{itemize}

Assume $e_1,\ldots,e_t$ are already claimed and $\P$ is updated accordingly.
Then, as next Maker chooses a free edge $e_{t+1}$ according to the following rules:

\begin{itemize}
\item[\textbf{R1.}] If there is $u \in U$ with $d_B(u,\End(\P_u\cup \P_f))\geq \delta n$, then $e_{t+1}$ is chosen such that it extends a path from $\P_u$ by the vertex $u$, if $|\P_u|\geq a+ 1$, or
a path from $\P_f$ by the vertex $u$, otherwise. 
	
\item[\textbf{R2.}] Otherwise, if there is a path $P \in \P_u$ with $d_{B}(v_1^P,U)+d_{B}(v_2^P,U)\geq 2$,
then there is an edge $v_i^Px\in H_t$ with $x\in U$, $i\in [2]$. Maker claims an arbitrary free edge $e_{t+1}=v_i^Pu$ with $u\in U$.

\item[\textbf{R3.}] Otherwise, if there is a path $P \in \P_f$ with $d_{B}(v_1^P,U)+d_{B}(v_2^P,U)\geq 2$, then there is an edge $v_i^Px\in H_t$ with $x\in U$, $i\in [2]$. Then
	\begin{enumerate}[a)]
	\item if there is a path $P_0\in \P_u$, Maker claims an arbitrary free edge $e_{t+1}\in 
		E(\End(P_0),x)$,
	\item otherwise, if $\P_u=\emptyset$, she claims an arbitrary free edge 
		$e_{t+1}=v_i^Pu$ with $u\in U$. 
	\end{enumerate}
	
\item[\textbf{R4.}] Otherwise, if there is $uw\in E_B(U)$, w.l.o.g.\
$d_{B}(w,\End(\P_u\cup \P_f))\leq d_{B}(u,\End(\P_u\cup \P_f))$, then Maker proceeds as follows.
	\begin{enumerate}[a)]
	
	\item If $d_{B}(u,\End(\P_u\cup \P_f))= d_{B}(w,\End(\P_u\cup \P_f)) 
	\geq \delta n-1$, then Maker chooses a free edge $e_{t+1}=ux$ with
	$x\in N_{B}(w,\End(\P_u\cup \P_f))$. Its existence is proved later.
	
	\item Otherwise, if $d_{B}(u,\End(\P_u\cup \P_f))\geq \delta n-1>
		d_{B}(w,\End(\P_u\cup \P_f))$, then let $P\in \P_u\cup \P_f$ be a path
		with $e_B(\End(P),u)=0$ and $d_B(v_i^P,U)=0$ for some $i\in [2]$. 
		Its existence is proved later. Maker then sets $e_{t+1}=uv_{3-i}^P$.
	
	\item Otherwise, if $d_{B}(u,\End(\P_u\cup \P_f)), d_{B}(w,\End(\P_u\cup 
		\P_f)) < \delta n-1$, let $P\in \P_u\cup P_f$, giving priority to
		unfinished paths, and let $d_{B}(v_i^P,U)\geq 
		d_{B}(v_{3-i}^P,U)$ for some $i\in [2]$. 
		Then Maker claims one of the edges $v_i^Pu$, $v_i^Pw$.
		
	\end{enumerate}
	
\item[\textbf{R5.}] Otherwise, in all the remaining cases, Maker extends a path which is not complete by some free edge $e_{t+1}$, arbitrarily, giving priority to unfinished paths.	
\end{itemize}

Stage I ends when after Maker's move, her graph consists only of finished paths, complete paths and isolated vertices. At this point, $(n(k-2))/(ka)+T/{a}$ rounds are played with $T=|P_c|$ 
being the number of complete paths at the end of Stage I.\\

\noindent \textbf{Stage II.} In the following $\left\lceil (n/k-T)/{a}\right\rceil-1$ rounds, Maker extends the finished paths. This time, Maker is interested in keeping the following property after each of her moves.
\begin{itemize}
\item[\textit{(F1)}] $\forall P \in \P_f: d_B(v_1^P,U)+d_B(v_2^P,U)\leq 1$.
\end{itemize}  
To describe Maker's strategy, we introduce the following terminology:
Let $e\in E(\End(\P_f),U)$. Then $e$ is called {\em good} if it is free;
otherwise we call it {\em bad}. We say that $P\in \P_f$ is a {\em bad path} if $d_B(v_1^P,U)+d_B(v_2^P,U)\geq 2$ holds, and with $\P_b\subseteq \P_f$ we denote the dynamic 
set of all bad paths. Moreover, we introduce the {\em potential}
$$\varphi := \sum\limits_{P\in \P_b} (d_B(v_1^P,U)+d_B(v_2^P,U)-1),$$
which measures dynamically
the number of edges that need to be deleted from $B$ in order to reestablish Property (F1).
Finally, with $e_{end}$ we denote the very last edge claimed by Maker in Stage II.

Now, in every round $i$ played in Stage II, Maker claims edges $e_1=e_1^{(i)},\ldots ,e_a=e_a^{(i)}$,
one after another. The $j^{\text{th}}$ edge $e_j$ is claimed according to the following rules:
\begin{itemize} \item Maker chooses a good edge $e_j$ between some vertex $u\in U$
and an endpoint of some path $P\in \P_f$ such that
\begin{itemize}
\item[(a)] in case $\varphi>0$, $\varphi$ is decreased after $e_j$ is claimed and all sets are updated,
\item[(b)] if $e_j\neq e_{end}$, then $\End(P)\cup \{u\}$ contains a vertex of the largest degree in Breaker's graph among all vertices from $\End(\P_f)\cup U$,
\item[(c)] if $e_j=e_a= e_{end}$, then after $e_j$ is claimed and all sets are updated,
there is a path $P\in \P_f$ with $e_B(\End(P),U)=0$.
\end{itemize}
	\item After $e_j$ is claimed, Maker removes $u$ from $U$ and $P$ from $\P_f$, and adds $P$ to $\P_c$, before she proceeds with $e_{j+1}$.
\end{itemize}

The exact details of how Maker finds such an edge $e_j$ will be given later in the proof.\\

{\bf Stage III.} Within one round, Maker claims at most $a$ free edges to complete a $P_k$-factor. The details are given later in the proof.\\

It is evident that if Maker can follow this strategy, she wins the game in the claimed number of rounds. For each of the stages above we show separately that Maker can follow her strategy.\\

%%%%%%%%%%%%%%%%%%%%%%
% STAGE I
%%%%%%%%%%%%%%%%%%%%%%

\noindent\textbf{Stage I.} We start with the following useful claim.

\begin{claim}
\label{clm:number_complete}
As long as Maker follows Stage I, $|P_c|\leq a+3/\delta < \delta n$.
\end{claim}

\begin{proof}
Following the strategy, Maker only creates complete paths in case there is a vertex of degree at least $\delta n -1$ in Breaker's graph 
which is used to extend a finished path (cases {R1}, {R4.a,b}),
or in case $\P_u=\emptyset$ (cases {R3.b}, {R4.c}, {R5}) holds.
The first option happens less than $3/\delta$ times, as Breaker claims less than
$n$ edges throughout Stage I. The second option can only happen in the last round of Stage I, 
which cannot lead to more than $a$ additional complete paths. 
\end{proof}

\medskip

Now, by induction on the number of rounds, $i$, we show that Maker can follow the proposed strategy of Stage I.
We first observe that before the game starts,
$(B,\P)$ is good, as $B$ is empty. Now, let us assume that she could
follow the strategy for the first $i-1$ rounds and that immediately after her $(i-1)^{\text{st}}$ move,
$(B,\P)$ is good. In particular, this also means that in the next round,
Property (Q4) is guaranteed for $t=0$, by choosing $H=H_0$ to be the graph of all the $a$ edges that Breaker claims in round $i$. By induction on the number of Maker's steps in round $i$, 
we prove that she can claim the edges $e_1,\ldots,e_a$ as described, and that
she ensures Property (Q4) for every $t\in \{0,\ldots,a\}$. Setting $t=a$ then
tells us that $(B,\P)$ is good immediately after Maker's move, completing the induction on $i$.\\

Let us assume that Maker already claimed $e_1,\ldots,e_t$ and that (Q4) holds after step $t$.
Let $H_t$ be the graph guaranteed by (Q4).
We now look at the different cases for step $t+1$.

\begin{itemize}
\item[\textbf{R1.}] In this case there must be an edge $g\in E(H_t)$ between $u$ and $\End(\P_u\cup \P_f)$, as $(B\setminus H_t,\P)$ satisfies Property $(Q1)$ after step $t$.
Now, if $|\P_u|\geq a+ 1$, then there needs to be a path $P\in \P_u$ with $d_B(v_1^P,U)+d_B(v_2^P,U)\leq 1$, as $(B\setminus H_t,\P)$ satisfies (Q3) after step $t$ and $e(H_t)\leq a$. Otherwise, Claim \ref{clm:number_complete} ensures that $|\P_f|\geq a+1$ and thus there is a path $P\in P_f$
with the same property. In either case, Maker can extend $P$ by $u$. Set $H_{t+1}:=H_t-g$.
Then, after the update, $g\in E(\End(\P))$ holds and therefore $g$ has no influence on (Q1)--(Q3)
anymore. Thus, using that $(B\setminus H_t,\P)$ satisfied (Q2) after step $t$, we conclude that $u$ has no edges towards $U$ in  $B\setminus H_{t+1}$ after step $t+1$. Now, one easily checks that $(B\setminus H_{t+1},\P)$ is good after step $t+1$.

\item[\textbf{R2.}] As $(B\setminus H_t,\P)$ satisfies (Q3) after step $t$, there needs to exist an edge $v_i^Px$ as 
	claimed. Moreover, we have $|U|\geq |\P_u\cup \P_f|= \frac{n}{k} - |P_c|>d_B(v_i^P,U)$
	where the last inequality follows from Claim \ref{clm:number_complete} and the fact
	that $(B\setminus H_t,\P)$ satisfies (Q3) after step $t$. 
	Thus, Maker can claim an edge $v_i^Pu$ as proposed. Afterwards, we set $H_{t+1}:=H_t-v_i^Px$.
	Then $u$ has no edge towards $U$ in $B\setminus H_{t+1}$ and the edge $v_i^Px$ has no influence 
	on (Q1)--(Q3) anymore. We therefore conclude that $(B\setminus H_{t+1},\P)$
	is good after step $t+1$.	

\item[\textbf{R3.}] 
The existence of $v_i^Px$ is given as in case {R2}.
If there is a path $P_0\in \P_u$, then not both edges $v_1^{P_0}x,v_2^{P_0}x$ can be claimed by Breaker, as otherwise $P_0$ would force case {R2}. 
If otherwise $\P_u=\emptyset$, then analogously to the argument in case {R2}, we have $|U|>d_B(v_i^P,U)$.
So, in either case, Maker can claim an edge as proposed by the strategy. After the update of $\P$, we obtain analogously to the previous case that $(B\setminus H_{t+1})$ is good with $H_{t+1}:=H_t-v_i^Px$.

\item[\textbf{R4.}] As case {R1} does not occur, we know that $d_{B}(u,\End(\P_u\cup \P_f)),d_{B}(w,\End(\P_u\cup \P_f))<\delta n$. In particular, a)--c) cover all possible subcases. Moreover,
as $(B\setminus H_t,\P)$ satisfies (Q2) after step $t$, we must have $uw\in E(H_t)$.
	\begin{enumerate}[a)]
	\item Assume there is no edge $ux$ as proposed by the strategy. Then there must be at least $
			\delta n-1$ vertices in $\End(\P_u\cup \P_f)$ which have degree at least $|\{u,w\}|=2$
			in $B$, contradicting the fact that $(B\setminus H_t,\P)$ satisfies (Q3) after step $t$
			and $e(H_t)\leq a$. So, Maker can claim an edge $ux$ as proposed.
			In case	$xw\in E(H_{t})$, we set $H_{t+1}:=H_t-xw$.		
			Then, after the update $xw\notin E(\End(\P)\cup U)$ holds, i.e.\ this edge has no 
			influence on (Q1)--(Q3). Moreover, $u$ has no edge towards $U$ in $B\setminus H_{t+1}$
			after step $t+1$. Otherwise, we set $H_{t+1}:=H_t-uw$. To see that again
			$(B\setminus H_{t+1},\P)$ is good after step $t+1$, just observe the following:
			After step $t+1$, $u$ has exactly one neighbour in $U$ (namely $w$) 
			in the graph $B\setminus H_{t+1}$. But, as $xw\in E(B\setminus H_t)$ 
			and (Q3) was fulfilled by $B\setminus H_t$ after step $t$, we know that the other 
			endpoint of the path $P$ has no edges towards $U$ in $B\setminus H_{t+1}$. 
			Moreover, $d_{B\setminus H_{t+1}}(w,\End(\P_f\cup \P_u))\leq 
			d_{B}(w,\End(\P_f\cup \P_u)) <\delta n$
			is maintained, as in $B$, $w$ gains $u$ and looses $x$ as a neighbour in 
			$\End(\P_f\cup \P_u)$.

	\item By Claim~\ref{clm:number_complete} and since $(B\setminus H_t,\P)$ is good after step $t$, 
		we have that 
		$|\P_u\cup \P_f|\geq \frac{n}{k}-\delta n>3a+d_{B\setminus H_t}(u,\End(\P_u\cup \P_f))\geq
		2a+d_{B}(u,\End(\P_u\cup \P_f))$. In particular, there are at least $2a$
		paths $P\in \P_u\cup \P_f$ with $e_B(\End(P),u)=0$.
		As $e(H_t)\leq a$ and since $(B\setminus H_t,\P)$ satisfied (Q3) after step $t$,
		there needs to be such a path with $d_B(v_i^P,U)=0$ for some $i\in [2]$. 
		Obviously Maker can claim the edge $v_{3-i}^Pu$, as $e_B(\End(P),u)=0$.
		Now, set $H_{t+1}:=H_t-uw$. Then to see that $(B\setminus H_{t+1},\P)$
		is good after step $t+1$, just notice that $u$ has only one edge, $uw$, towards 
		$U$ in $B\setminus H_{t+1}$, while $d_{B\setminus H_{t+1}}(v_i^P,U)=0$. Moreover,
		$d_B(w,\End(\P_u\cup \P_f))<\delta n$ is guaranteed by the assumption 
		on $w$ for this case and since $w$ gains at most one new neighbour among $\End(\P)$, namely 
		$u$.		
		
		\item As the cases {R2} and {R3} do not occur, we have
		$d_B(v_{3-i}^P,U)=0$ and $d_B(v_{i}^P,U)\leq 1$. In particular, one of the edges
		$v_i^Pu,v_i^Pw$ is free, so Maker can follow the proposed strategy.
		W.l.o.g.\ let Maker claim $v_i^Pu$. Set $H_{t+1}:=H_t-uw$.
		Then after step $t+1$, $u$ has exactly one edge towards $U$ (namely $uw$) 
		in $B\setminus H_{t+1}$, while $d_{B\setminus H_{t+1}}(v_{3-i}^P,U)=0$.
		Moreover, $d_B(w,\End(\P_u\cup \P_f))<\delta n$ is guaranteed as
		in the previous case. Therefore, we conclude analogously that $(B\setminus H_{t+1},\P)$
		is good after step $t+1$.
 	\end{enumerate}

\item[\textbf{R5.}] As the cases {R1}--{R4} do not occur,
$(B,\P)$ is good after step $t$. Therefore, Maker can easily follow the proposed
strategy and afterwards $(B\setminus H_{t+1},\P)$ is good for every graph $H_{t+1}$.
\end{itemize}
So, in either case Maker can follow the proposed strategy for Stage I.\\

{\bf Stage II.} At any moment throughout Stage II, let $p:=|\P_f|$ and let $br$ denote the number of bad edges. By the definition of $\varphi$, it always holds that $br\leq p+\varphi$. Moreover,
before every move of Maker in this stage, we have $p\geq a+1$, and therefore $p\geq 2$ holds immediately before $e_{end}$ is claimed.\\

By induction on the number of rounds, $i$, we now prove that Maker can follow the proposed strategy
and always maintains (F1). We can assume that (F1) was already satisfied after Maker's last move in Stage I, using (Q3). Now, assume that Maker's $i^{\text{th}}$ move happens in Stage II. As, by induction, (F1) was satisfied immediately after her previous move and as
Breaker afterwards claimed only $a$ edges in his previous move, we know that
$\varphi\leq a<p$ before Maker's move, and therefore $br\leq p+\varphi<2p$. 
We now observe that such a relation can be maintained as long as Maker can follow her strategy.

\begin{claim}\label{clm:StageII}
Assume that Maker can follow her strategy. Then, after $e_j$ is claimed and $\P_f, U$ are updated accordingly, $\varphi<p$ and $br<2p$ hold.
\end{claim}

%\begin{proofof}{Claim~\ref{clm:StageII}}
\begin{proof}
For induction assume that $\varphi<p$ and $br<2p$ hold immediately after $e_{j-1}$ is claimed. When Maker claims $e_j$, two cases may occur. If $\varphi=0$ holds, then after $e_j$ is claimed, we still have $\varphi=0$ and $br\leq p+\varphi<2p$. Otherwise, we have $\varphi>0$, in which case Maker claims an edge that decreases the value of $\varphi$. As $p$ decreases by one within one step of Maker, we thus obtain that $\varphi<p$ and $br\leq p+\varphi<2p$ are satisfied after all updates.
\end{proof}

\medskip

With this claim in hand, we can deduce that Maker can always follow the proposed strategy. 

Consider first that $e_j\neq e_{end}$. We note that each path in $\P_f$ and
each vertex in $U$ is incident with $2p>br$ edges from $E(\End(\P_f),U)$.
Thus each such path and each such vertex intersects at least one good edge.
Let $v$ be a vertex of the largest Breaker's degree among all vertices in $\End(\P_f)\cup U$. 

Assume first that $v\in \End(P)$ for some $P\in\P_f$. If $\varphi=0$ or $P\in P_b$, then Maker claims an arbitrary good edge $e_j\in E(\End(P),U)$ which exists as explained above. Just note that in case $P\in P_b$, the value of $\varphi$ will be decreased, as $P$ gets removed from $\P_b$. Moreover,
$v$ is contained in the path $P$ (after the update).
Otherwise, if $\varphi>0$ and $P\notin P_b$, we find some bad path $P_0\neq P$ and some bad edge $u_0v_0$ with $u_0\in U$ and $v_0\in\End(P_0)$.
Then Maker claims an edge $e_j=xu_0$, with $x\in \End(P)$ and $d_B(x,U)=0$. Such a vertex exists, as we assumed $P\notin P_b$. Again this decreases $\varphi$ by (at least) one, as 
$d_B(v_1^{P_0},U)+d_B(v_2^{P_0},U)$ decreases by (at least) one after $u_0$ is removed from $U$;
and again $v$ is contained in the updated path $P$. 

Assume then that $v\in U$. If $\varphi=0$, then Maker can choose $e_j$ to extend an arbitrary path in $\P_f$ by the vertex $v$, which is possible as there are no bad paths. So, assume that there is a path $P\in P_b$. If there is a good edge in $E(v,\End(P))$,
Maker claims such an edge and then $\varphi$ decreases as $P$ is removed from $\P_f$, and
$v$ again is contained in the updated path $P$.
If there is no such good edge, then Maker claims an arbitrary good edge
between $v$ and some path $P_0\in \P_f\setminus \{P\}$. Then, $\varphi$ decreases as
$e_B(v,\End(P))=2$ and $v$ gets removed from $U$, and $v$ finally belongs to the updated path $P_0$.\\
Consider then that $e_j=e_a=e_{end}$, and recall that $p\geq 2$ before Maker claims $e_{end}$.
We know that $\varphi\leq a$ holds before Maker's
first step in round $i$. As Maker decreased the value of $\varphi$ by at least one,
in case $\varphi>0$, with every previous edge in this round, we know that $\varphi\leq 1$
immediately before she wants to claim $e_{end}$. W.l.o.g.\ let $\varphi =1$, and let
$P_0$ be the unique bad path. Note that then $e_B(\End(P_0),U)=2$.
Moreover, let $P\in \P_f\setminus \{P_0\}$. If $e_B(\End(P),U)=0$,
then Maker extends $P_0$ by an arbitrary edge $e_{end}$ (which is possible as $br<2p$).
Then, after the update, $\varphi=0$ holds as $P_0$ is removed from $\P_f$, and $P$ satisfies 
$e_B(\End(P),U)=0$. Otherwise, we have $e_B(\End(P),U)=1$ as $P\notin \P_b$,
and thus there is a unique vertex $u$ such that
$e_B(\End(P),u)=1$. If $e_B(\End(P_0),u)\leq 1$ holds, then Maker claims an edge
$e_{end}=uv_i^{P_0}$ with $i\in [2]$. Afterwards, $P$ satisfies $e_B(\End(P),U)=0$
as $u$ gets removed from $U$;
$\varphi=0$ holds as $P_0$ is removed from $\P_b$. Otherwise, we
have $e_B(u,\End(P_0))=2=e_B(U,\End(P_0))$ as $\varphi=1$. In this case Maker just 
claims a free edge $e_{end}=uv_i^P$ with $i\in [2]$ (which is possible as
$e_B(\End(P),u)=1$). Then, $P_0$ satisfies
$e_B(\End(P_0),U)=0$ after $u$ is removed from $U$, and as $P_0$ is not bad anymore, i.e.\ $\varphi=0$.\\

In total, we see that in either case Maker can claim the edges $e_j$ as proposed. Finally,
Property (F1) always holds immediately after $e_a$ is claimed. For this just recall that $\varphi\leq a$ holds immediately before a Maker's move, and that Maker reduces $\varphi$ by at least one in each step as long as $\varphi>0$ holds. That is, we obtain $\varphi=0$ at the end of her move, which makes Property (F1) hold.\\

{\bf Stage III.} Finally, we prove that Maker can finish a $P_k$-factor within one additional round.
We start with the following claim.

\begin{claim}
\label{clm:StageIII}
Before Maker's move in Stage III, $d_B(v)<2\delta n$
holds for every $v\in (U \cup End(\P_f))$.
\end{claim}

%\begin{proofof}{Claim~\ref{clm:StageIII}}
\begin{proof}
Suppose that the statement does not hold, i.e.\ there exists a vertex $ v\in U \cup End(\P_f)$
such that $d_B(v)\geq 2\delta n$. Then, during the last $\left\lfloor 3/\delta \right\rfloor$ rounds
of Stage II, $d_B(v)>\delta n$. However, in each step of these rounds (except
when claiming $e_{end}$), Maker included a vertex $w$ into some
complete path for which $d_B(w)\geq d_B(v)>\delta n$ was satisfied (see (b)). But then,
Breaker would have claimed more than $n$ edges, a contradiction to the number
of edges he could claim in all rounds so far. 
\end{proof}

\medskip

With this claim in hand, we are able to describe how to finish a $P_k$-factor within one further round. 
We distinguish between the following cases depending on the size of $U$.\\

\textbf{Case 1: $\mathbf{0<|U|\leq a/2}$.} We denote the isolated vertices in $U$ by $u_1,u_2,\ldots, u_t$, and the paths in $\P_f$ by $P_1\dots,{P_t}$. Using Claim \ref{clm:StageIII},
for every $i\in [t]$, we find at least $|\P_c|-4\delta n>n/(4k)$ paths $R\in \P_c$ such that
$u_iv_1^{R}, v_2^{R}v_1^{P_i}$ are free. We thus can fix $t$ distinct paths $R_1,\ldots,R_t\in \P_c$ such that, for every $i\in [t]$, the edges $u_iv_1^{R_i}, v_2^{R_i}v_1^{P_i}$ are free. Maker claims these edges, in total at most $a$, and by this completes a $\pkf$, as $V(R_i)\cup V(P_i)\cup \{u_i\}$ contains a copy of $P_{2k}$ for every $i\in [t]$.

\textbf{Case 2: $\mathbf{a/2<|U|=a}$.} Let $\G=(\P_f\cup U,E(\G))$ be the bipartite graph where
two vertices $u\in U$ and $P\in \P_f$ form an edge if and only if $d_B(u,\End(P))\leq 1$.
Then $e(\overline{\cal G})\leq a-1$ holds, as after Maker's last move in Stage II
we had $e_B(\End(\P_f),U)\leq a-1$ (by (F1) and (c)), 
while Breaker afterwards claimed at most $a$ bad edges. 
By the theorem of K\"onig-Egervary (see e.g. \cite{West}) we thus obtain that $\cal G$
contains a matching of size at least
$$\frac{|U|^2-e(\overline{\cal G})}{|U|}>\begin{cases}
|U|-1, & \text{if $|U|=a$}\\
|U|-2, & \text{if $a/2<|U|<a$.}\\
\end{cases}$$
So, in case $|U|=a$, we have that ${\cal G}$ contains a perfect matching,
say $u_1P_1, u_2P_2,\ldots, u_aP_a$. Then Maker claims a good edge in $E(u_i,\End(P_i))$
for every $i\in [a]$, and by this creates a copy of $\pkf$.
Otherwise, in case $a/2<|U|\leq a-1$, we find a matching of size $|U|-1$,
say $u_1P_1, u_2P_2,\ldots, u_{|U|-1}P_{|U|-1}$. Maker then claims
a good edge in $E(u_i,\End(P_i))$ for every $i\in [|U|-1]$, in total at most $a-2$ edges.
For the remaining (unique) vertices $u\in U$ and $P\in P_f$ that are not covered by the matching,
we proceed as in Case 1: we find a path $R\in \P_c$ such that the edges $uv_1^{R}, v_2^{R}v_1^{P}$ are free, which then Maker claims to complete a $P_k$-factor.
\end{proofof}

 % Fast P_k-factor Game
\section{Weak $S_k$-factor game}
\label{wSkf}
\begin{proofof}{Theorem~\ref{weakSk}} 

In the following we give a strategy for Maker in the $\skf$ game. Afterwards,
we prove that she can follow that strategy and win in the claimed number of rounds.\\

\textbf{Maker's strategy.} Maker makes the $k$-star factor by gradually increasing the size of 
$n/k$ stars, so that at any point of the game, no two disjoint stars differ in size by more than one. Before the game starts, she splits the vertex set into three sets $C$, $R$ and $F$, that are dynamically maintained. %during this stage. 
$C$ represents the centres of the stars in the star factor, $F$ contains the endpoints of the current stars, and $R$ are the remaining isolated vertices in Maker's graph. At the beginning of the game, $F:=\emptyset$, $C$ contains $n/k$ arbitrary chosen vertices and $V=C \cup R$. All the star centres have degree $0$ at the beginning of the game. Maker's strategy consists of the following two stages.

\medskip

\textbf{Stage I.} 
Maker divides this stage into phases $1,2,\dots, k-1$. 
The game starts in phase $1$, when all vertices in $C$ have degree $0$ in Maker's graph. 
In phase $i$, $1\leq i \leq k-1$, she makes the vertices in $C$ get degree exactly $i$ in her graph.  
The phase $i$, $1\leq i \leq k-2$ finishes (and the phase $i+1$ starts) immediately after the step 
in which the last vertex of $C$ reached degree $i$. So, it might happen that Maker switches from phase $i$ to phase $i+1$ between two steps of the same move.
In particular, she will play exactly $n/k$ consecutive steps (from consecutive rounds) 
in each of her first $k-2$ phases. Finally, Stage I ends immediately after the Maker's move 
in which phase $k-2$ ended, 
i.e.\ when all the vertices in $C$ have degree at least $k-2$ in Maker's graph. (Note that
it might happen that phase $k-1$ consists of zero steps.)

\medskip

To describe Maker's strategy more precisely, we let $C_A$ always denote the subset of $C$ 
containing the vertices of smallest degree in her graph. That is, in phase $i$, $C_A$ contains those vertices that are centres of stars of size $i-1$.
At the beginning of each phase, $C_A=C$. Moreover, we call a Breaker's edge $e$ {\em bad} if $e\in E(C,R).$

\medskip

Assume now, Maker wants to make her $j^{th}$ move in Stage I. 
Let $t$ be the number of elements in $C_A$ right at the beginning of her move.
Maker iteratively chooses the edges $e_1=e_1^{(j)},\ldots, e_a=e_a^{(j)}$ 
of her $j^{\text{th}}$ move in the following way: For every $1\leq s\leq a$, she first sets $t:=|C_A|$. Then,
\begin{enumerate}[(1)]
\item if there is a free edge $e_F\in E(C_A,R)$ such that $\emptyset\neq e_F\cap e_B\in R$ for some bad edge $e_B,$
then Maker chooses $e_s$ to be such an edge $e_F$. Let $x_s\in C_A$ and $y_s\in R$ be the vertices of $e_s$. 
\item Otherwise, she chooses $e_s=x_sy_s$ arbitrarily
with $x_s\in C_A$ and $y_s\in R$.
\item Afterwards, she updates $R:=R\setminus \{y_s\}$, $F:=F\cup \{y_s\}$ and 
$C_A:=C_A\setminus \{x_s\}$ if $t\neq 1$, or $C_A:=C$ if $t= 1$,
before she proceeds with $e_{s+1}$.
\end{enumerate}

This stage lasts $\left\lceil (k-2)n/(ak)\right\rceil$ rounds. \\

\noindent \textbf{Stage II.} When Maker enters Stage II, her graph consists of stars of size $k-2$ and $k-1$.
Moreover, $$|C_A|=|R|=\frac{(k-1)n}{k}-a\left\lceil\frac{(k-2)n}{ak}\right\rceil =:N.$$ Maker now
completes her $k$-star factor, by claiming a perfect matching between $C_A$ and $R$
in the following $\lfloor N/a \rfloor +1$ rounds. The details follow later in the proof.

\medskip

It is easy to see that if Maker can follow the strategy, she wins the $\skf$ within the claimed number of rounds. 
Indeed, the number of rounds Stage I and II last together is 
\begin{align*}
\left\lceil\frac{(k-2)n}{ak}\right\rceil + \lfloor N/a \rfloor +1 
	= \left\lfloor \frac{(k-1)n}{ak} \right\rfloor +1
	= \begin{cases}
	\lceil \frac{(k-1)n}{ak} \rceil, & \text{ if }ak\nmid (k-1)n\\
	\lceil \frac{(k-1)n}{ak} \rceil +1, & \text{ otherwise.}\\
	\end{cases}
\end{align*}
So, to finish the proof, we show separately for each stage that Maker can follow the proposed strategy.\\

\noindent \textbf{Stage I.}  
In order to show that Maker can follow her strategy in this stage, we first prove that 
as long as she can follow the strategy, the number $e_B(C,R)$ of bad edges cannot be too large.
Based on this, we then conclude that Maker indeed can follow her strategy. 

Before this, it is useful to observe the following:
If $|C_A|>e_B(C,R)$ holds before Maker claims an edge, 
we know that each vertex in $R$ has a free neighbour
in $C_A.$ Thus, Maker can easily claim an edge according to $(1)$
if there exists some bad edge. Moreover, this bad edge
then disappears from the set of bad edges when $F$ and $R$ are updated.
Thus, the number of bad edges decreases,
and again $|C_A|>e_B(C,R)$ holds. 
In particular, we can continue this way. So, if $|C_A|>e_B(C,R)=:b'$ holds 
at the beginning of a Maker's move,
then Maker in her whole move decreases the number of bad edges 
by at least $\min\{a,b'\}$. Using this,
we prove our first claim.

\begin{claim}
\label{cl:constBound}
Assume that Maker can follow the proposed strategy.
Then, for each $1\leq i\leq k-2$,
there exists a constant upper bound $c=c(a,i)$ for the number of bad edges 
throughout the phase $i$. 
\end{claim}

\begin{proof}
The proof goes by induction on $i$. Set $c(a,0):=0$.
There can be at most $c(a,i-1)+a$ bad edges before the first Maker's move
that happens completely in phase $i$, either by induction hypothesis (when $i>1$) or as Breaker claims $a$ edges in his first move (when $i=1$). Now, by the observation above, we know that as long as $|C_A|> e_B(C,R)$, Maker can reduce the number of bad edges by (at least) $\min\{e_B(C,R),a\}$ in each round, 
while Breaker can increase this number
by at most $a$. Thus, as long as $|C_A|>c(a,i-1)+a$ holds before some Maker's move, 
her strategy ensures that after such a move, $e_B(C,R)\leq c(a,i-1)$ holds.
Just when $|C_A|\leq c(a,i-1)+a$ holds before a move of Maker, it might happen that Maker cannot reduce the number of bad edges anymore. 
But then, the number of remaining steps in phase $i$ is bounded by a constant,
and Breaker can add only another constant number of bad edges until phase $i$ ends, giving some constant bound $c(a,i)$. Thus, before any Maker's move in phase $i$ there cannot be more than $c(a,i)$ bad edges, which completes the claim.
\end{proof}

\begin{claim}
For large $n$, Maker can follow the strategy of Stage I.
\end{claim}

\begin{proof}
Assume that Maker could follow the strategy for the first $j$ rounds.
Our goal is to show that she can do so in round $j+1$ as well.
For this, observe that before Maker's move, $|R|\geq n/k$ holds.
Moreover, the number of bad edges is bounded by a constant, 
according to the previous claim. So, provided $n$ is large enough, 
there exist more than $a$ vertices in $R$ that 
are not incident with bad edges.
In particular, in each step of her $(j+1)^{\text{st}}$ move,
Maker can claim an edge according to $(2)$. Thus,
she can follow the strategy. 
\end{proof}

\medskip

\textbf{Stage II.} Observe that when Maker enters Stage II, the number of bad edges 
is at most $c(a,k-2)+a$, according to the Claim~\ref{cl:constBound}.
Moreover, $N=|C_A|=|R|\geq n/k-a$.
Provided that $n$ is large enough, Proposition~\ref{WeakPMstronger} (with $G$ being the bipartite graph induced by the free edges
between $C_A$ and $R$)
now ensures that Maker has a strategy to create 
the desired matching within $\lfloor N/a\rfloor +1$ rounds.
\end{proofof} % Fast S_k-factor Game
\section{Conculding remarks and open problems}

{\bf Star factor game.} Theorem \ref{weakSk} tells us that
$\tau_{\skf}(a:a)\in \{ (k-1)n/(ka),(k-1)n/(ka)+1\}$ for large enough $n$, in case $ak|(k-1)n$. In fact, in can be checked that there are pairs $(a,k)$ were the first value occurs, while there exist pairs $(a,k)$ for which it does not. It would be interesting to describe all the pairs $(a,k)$ for which Maker cannot win the $(a:a)$ $\skf$-game perfectly fast and to determine a winning strategy for the first player in the corresponding strong version in these cases.\\

{\bf $H$-factors.} More generally, it seems to be challenging to describe
$\tau_{\F}(a:a)$ in case $\F$ is the family of $H$-factors, for any given graph $H$ not being a forest.
Even in the case $a=1$ not so much is known. In all the games $\F$ we studied here,
it happens that $\tau_{\F}(a:a)=(1+o(1))\tau_{\F}(1:1)/a$. We wonder whether there exist families $\F$ of spanning subgraphs of $K_n$, where such a relation does not hold.
 %Concluding remarks and open problems

 % Bibliography


\begin{thebibliography}{99}
%\thispagestyle{empty}\markboth{Bibliography}{Bibliography}
%\addcontentsline{toc}{chapter}{Bibliography}

\bibitem{BMP}
%J. Balogh, R. Martin, and A. Pluh\'ar, The diameter game, {\em
%Random Structures and Algorithms} 35 (3) (2009), 369--389.
%
%\bibitem{Beckbook}
%J. Beck, 
%{\bf Combinatorial Games: Tic-Tac-Toe Theory},
%Encyclopedia of Mathematics and Its Applications 114, Cambridge University Press, 2008.
%
%\bibitem{FH}
%A. Ferber and D. Hefetz, 
%{\em Winning strong games through fast strategies for weak games},
%Electronic Journal of Combinatorics 18(1) (2011), P144.
%
%\bibitem{FHcon}
%A. Ferber and D. Hefetz, 
%{\em Weak and strong $k$-connectivity games},
%The European Journal of Combinatorics (2014), 169--183. 
%
%\bibitem{GeSa}
%H.\ Gebauer and T.\ Szab\'{o}, Asymptotic random graph intuition for the
%biased connectivity game, {\em Random Structures and Algorithms} 35
%(2009), 431--443.
%
%\bibitem{HKSS09}
%D. Hefetz, M. Krivelevich, M. Stojakovi\'c and T. Szab\'o, 
%{\em Fast winning strategies in Maker-Breaker games}, 
%Journal of Combinatorial Theory Series B 99 (2009), 39--47.
%
%\bibitem{HMS12} D.\ Hefetz, M.\ Mikala\v{c}ki and M.\ Stojakovi\'{c}, Doubly biased Maker-Breaker Connectivity game, \textit{The Electronic Journal of Combinatorics} 19 (1) (2012), P61.
%
%\bibitem{HS09}
%D. Hefetz and S. Stich, 
%{\em On two problems regarding the Hamilton cycle game}, 
%Electronic Journal of Combinatorics 16 (1) (2009), R28.
%
%\bibitem{Lehman}
%A. Lehman, A solution of the Shannon switching game, {\it J. Soc. Indust. Appl. Math.} 12
%(1964), 687--725.

J.\ Balogh, R.\ Martin, and A.\ Pluh\'ar, The diameter game, {\em
Random Structures and Algorithms} 35 (3) (2009), 369--389.

\bibitem{Beckbook}
J. Beck, 
{\bf Combinatorial Games: Tic-Tac-Toe Theory},
Encyclopedia of Mathematics and Its Applications 114, Cambridge University Press, 2008.

\bibitem{CFKL}
D.\ Clemens, A.\ Ferber, M.\ Krivelevich and A.\ Liebenau,
Fast strategies in Maker-Breaker games played on random boards,
{\em Combinatorics, Probability and Computing} 21 (2012), 897--915.

\bibitem{FH}
A.\ Ferber, and D.\ Hefetz, 
Winning strong games through fast strategies for weak games,
{\em The Electronic Journal of Combinatorics} 18(1) (2011), P144.

\bibitem{FHcon}
A.\ Ferber, and D.\ Hefetz, 
Weak and strong $k$-connectivity games,
{\em The European Journal of Combinatorics} (2014), 169--183. 

\bibitem{GeSa}
H.\ Gebauer, and T.\ Szab\'{o}, Asymptotic random graph intuition for the
biased connectivity game, {\em Random Structures and Algorithms} 35
(2009), 431--443.

\bibitem{HKSS09}
D.\ Hefetz, M.\ Krivelevich, M.\ Stojakovi\'c, and T.\ Szab\'o, 
Fast winning strategies in Maker-Breaker games, 
{\em Journal of Combinatorial Theory Series B} 99 (2009), 39--47.

\bibitem{HKSSBook}
D.\ Hefetz, M.\ Krivelevich, M.\ Stojakovi\'c, and T.\ Szab\'o, {\bf Positional games}, Oberwolfach Seminars 44, 
Birkh\"{a}user, 2014.


\bibitem{HMS12} D.\ Hefetz, M.\ Mikala\v{c}ki, and M.\ Stojakovi\'{c}, Doubly biased Maker-Breaker Connectivity game, {\em The Electronic Journal of Combinatorics} 19 (1) (2012), P61.

\bibitem{HS09}
D.\ Hefetz, and S.\ Stich, 
On two problems regarding the Hamilton cycle game, 
{\em The Electronic Journal of Combinatorics} 16 (1) (2009), R28.

\bibitem{Lehman}
A.\ Lehman, A solution of the Shannon switching game, {\em J. Soc. Indust. Appl. Math.} 12
(1964), 687--725.

\bibitem{Posa}
 L.\ P\'{o}sa, Hamiltonian circuits in random graphs, {\em Discrete Mathematics} 14 (1976) 359--364.

\bibitem{West}
D.\ B.\ West, {\bf Introduction to Graph Theory}, Prentice Hall,
2001.
\end{thebibliography}
\end{document}